\documentclass[12pt]{amsart}

\usepackage{amssymb}
\usepackage{latexsym}
\usepackage{verbatim,longtable}
\usepackage{fullpage,color}
\usepackage{euscript}

\input {cyracc.def}
\numberwithin{equation}{section}

\usepackage[colorlinks=true,linkcolor=blue,urlcolor=my_color,citecolor=magenta]{hyperref}

 \font\tencyr=wncyr10 
\font\tencyi=wncyi10 
\font\tencysc=wncysc10 
\def\rus{\tencyr\cyracc}
\def\rusi{\tencyi\cyracc}
\def\rusc{\tencysc\cyracc}

\newtheorem{thm}{Theorem}[section]
\newtheorem{conj}[thm]{Conjecture}
\newtheorem{lm}[thm]{Lemma}
\newtheorem{cl}[thm]{Corollary}

\theoremstyle{remark}
\newtheorem{rmk}[thm]{Remark}
\newtheorem*{rem}{Remark}

\theoremstyle{definition}
\newtheorem{ex}[thm]{Example}
\newtheorem{df}{Definition}

\newenvironment{E6}[6]{%
{\small\begin{tabular}{@{}c@{}}
{#1}--{#2}--\lower3.5ex\vbox{\hbox{{#3}\rule{0ex}{2.5ex}}
\hbox{\hspace{0.4ex}\rule{.1ex}{1ex}\rule{0ex}{1.4ex}}\hbox{{#6}\strut}}--{#4}--{#5}
\end{tabular}}}

\newcommand {\ah}{{\mathfrak a}}
\newcommand {\be}{{\mathfrak b}}

\newcommand {\g}{{\mathfrak g}}
\newcommand {\h}{{\mathfrak h}}

\newcommand {\n}{{\mathfrak n}}

\newcommand {\es}{{\mathfrak s}}
\newcommand {\te}{{\mathfrak t}}
\newcommand {\ut}{{\mathfrak u}}

\newcommand {\sln}{\mathfrak{sl}_n}

\newcommand {\co}{{\mathcal O}}

\newcommand {\BZ}{{\mathbb Z}}
\newcommand {\BN}{{\mathbb N}}
\newcommand {\BR}{{\mathbb R}}

\newcommand{\bow}{\boldsymbol{W}}
\newcommand {\esi}{\varepsilon}
\newcommand{\lb}{\lambda}
\newcommand{\ap}{\alpha}
\newcommand{\vp}{\varphi}

\renewcommand{\le}{\leqslant}
\renewcommand{\ge}{\geqslant}
\newcommand{\curge}{\succcurlyeq}
\newcommand{\curle}{\preccurlyeq}

\newcommand {\sfr}{\eus R}
\newcommand{\eus}{\EuScript}

\newcommand {\Aut}{{\mathsf{Aut}}}

\newcommand {\hot}{{\mathsf{ht}}}

\newcommand {\Ima}{{\mathsf{Im}}}

\newcommand {\ord}{{\mathsf{ord}}}

\newcommand {\rk}{{\mathsf{rk\,}}}

\newcommand {\AN}{\mathfrak{An}}
\newcommand {\anp}{\AN(\eus P)}
\newcommand {\AND}{\AN(\Delta^+)}

\newcommand {\anod}{\AN(\Delta(1)\!)}
\newcommand {\mdt}{\eus M_{\Delta(1)}(t)}
\newcommand {\ndt}{\eus N_{\Delta(1)}(t)}

\newcommand {\fX}{\mathfrak X}
\newcommand {\xdo}{\mathfrak X_{\Delta(1)}}

\newcommand {\tri}{\mathfrak{sl}_2}
\newcommand {\GR}[2]{{\textrm{{\bf #1}}}_{#2}}
\newcommand {\GRt}[2]{{\tilde{\textrm{{\bf #1}}}}_{#2}}

\newcommand {\ov}{\overline}

\newcommand {\beq}{\begin{equation}}
\newcommand {\eeq}{\end{equation}}

\definecolor{my_color}{rgb}{0,0.5,0.5}
\definecolor{MIXT}{rgb}{0.7,0.1,0.2}
\definecolor{MIX}{rgb}{0.4,0.3,0.6}

\begin{document}
\hfill {{
\scriptsize November 27, 2014}}
\vskip1ex

\title[Antichains in weight posets]
{{\color{MIX}Antichains in weight posets associated with gradings of simple Lie algebras}}
\author[D.\,Panyushev]{Dmitri I.~Panyushev}
\address{Institute for Information Transmission Problems of the R.A.S., 
\hfil\break\indent B. Karetnyi per. 19, Moscow 
127994, Russia}
\email{panyushev@iitp.ru}
\address[]{Independent University of Moscow,
Bol'shoi Vlasevskii per. 11, 119002 Moscow, \ Russia}
\begin{abstract}
For a reductive Lie algebra $\h$ and a simple finite-dimensional $\h$-module $V$, the set of 
weights of $V$, $\eus P(V)$, is equipped with a natural partial order.  We consider antichains 
in the weight poset $\eus P(V)$ and a certain operator $\fX$ acting on antichains. Eventually, we impose 
stronger constraints on $(\h,V)$ and stick to the case in which $\h=\g(0)$ and $V=\g(1)$ for 
a $\BZ$-grading $\g=\bigoplus_{i\in\BZ}\g(i)$ of a simple Lie algebra $\g$. Then $V$ is a weight 
multiplicity free $\h$-module and $\eus P(V)$ can be regarded as a 
subposet of $\Delta^+$, where $\Delta$ is the root system of $\g$. Our goal is to demonstrate that the weight posets associated with $\BZ$-gradings exhibit many good properties that are similar to those of 
$\Delta^+$ that are observed earlier in \cite{oper-X}.
\end{abstract}
\maketitle

\tableofcontents

\section{Introduction}  
\label{sec:intro}

\noindent
Let $(\eus P, \curle)$ be a finite poset and $\AN(\eus P)$ the set of all antichains in $\eus P$.
If $\Gamma\in\AN(\eus P)$, then 
\begin{gather*}
I_+(\Gamma)=I(\Gamma):=\{\nu\in\eus P\mid \gamma\curle\nu \text{ for some } \gamma\in\Gamma\} 
\ \text{  and}  \\ 
I_-(\Gamma):=\{\nu\in\eus P\mid \nu\curle\gamma \text{ for some } \gamma\in\Gamma\} 
\end{gather*}
We say that $I(\Gamma)$ (resp. $I_-(\Gamma)$) is  the {\it upper} (resp.  {\it lower}) {\it ideal of\/} 
$\eus P$ generated by $\Gamma$. For any $M\subset \eus P$, let $\min(M)$ (resp. $\max(M)$) denote 
the set of minimal (resp. maximal) elements of $M$ with respect to `$\curle$'. 
Then  $\Gamma=\min(I(\Gamma))=\max(I_-(\Gamma))=I(\Gamma)\cap I_-(\Gamma)$ and the above
formulae  provide one-to-one correspondences between the antichains, upper ideals, and lower ideals of $\eus P$.
\begin{df}
The  {\it reverse operator\/} associated with $\eus P$, $\fX=\fX_\eus P: \anp\to \anp$, is defined by $\fX(\Gamma)=\max(\eus P\setminus I(\Gamma))$. 
\end{df}
\noindent
It is easily seen that $\fX\bigl(\min(\eus P\setminus I_-(\Gamma))\bigr)=\Gamma$,
hence $\fX$ is invertible.
The {\it order\/} of $\fX$, denoted $\ord(\fX)$, is the order of the cyclic group $\langle\fX\rangle$
generated by $\fX$ in the permutation group on $\anp$. By definition, an $\fX$-{\it orbit\/} is an orbit of the group $\langle\fX\rangle$ in $\anp$. 
The early history of the reverse operators, which goes back to 1970's, is related to the special case in 
which $\eus P$ is the Boolean algebra $\mathbb B^n$, see \cite{deza} and references therein. 
The study of $\fX$ in the general setting was initiated in~\cite{flaas1,flaas2}.

For a natural class of posets $\eus P$, we wish to know
$\#\anp$, $\max_{\Gamma\in\anp}(\#\Gamma)$, $\ord(\fX)$, and other properties of the $\fX$-orbits in
$\anp$.
It is also of interest to  determine the following  refinements ($t$-analogues) of $\#\anp$:

1) \ Set $\eus N_\eus P(t)=\sum_{\Gamma\in\anp} t^{\#\Gamma}=\sum N_it^i$. Here the coefficient of 
$t^i$ is the number of antichains of size $i$. Consequently, $\eus N'_\eus P(1)/\eus N_\eus P(1)$ 
is the average value of the size of all antichains in $\eus P$. It is clear that $N_1=\#\eus P$ and 
$\deg\eus N_\eus P(t)=\max_{\Gamma\in\anp}(\#\Gamma)$.

2) \ Set $\eus M_\eus P(t)=\sum_{\Gamma\in\anp} t^{\#I(\Gamma)}=\sum M_it^i$. Here the coefficient of 
$t^i$ is the number of upper ideals of cardinality $i$. Consequently, 
$\eus M'_\eus P(1)/\eus M_\eus P(1)$ is the average value of the cardinality of all upper ideals.
It is clear that  $\eus M_\eus P(t)$ is monic and $\deg\eus M_\eus P(t)=\#\eus P$.
\par
These two $t$-analogues of $\anp$ will be referred to as {\it the\/ $\eus M$- and $\eus N$-polynomials\/} of $\eus P$.

\noindent
Let $\Delta^+$ be a set of positive roots of a reduced irreducible root system $\Delta$. Then $\Delta^+$
is a graded poset, and there is a number of striking results on $\AND$ and $\fX_{\Delta^+}$.
We refer to \cite{bour,hump} for basic definitions and properties of root systems.
Let $\Pi=\{\ap_1,\dots,\ap_n\}$ be the set of {\it simple roots\/} in $\Delta^+$,
$W$ the {\it Weyl group\/}, and $h$ the {\it Coxeter number\/} of $\Delta$. 
The partial order in $\Delta^+$ is 
defined by the condition that $\gamma$ covers $\mu$ if and only if $\gamma-\mu\in\Pi$. 
For $\gamma\in\Delta$, let $[\gamma:\ap_i]$ be the coefficient of $\ap_i$ 
in the expression of $\gamma$ via the simple roots.
The {\it height\/} of $\gamma$ is $\hot(\gamma)=\sum_{i=1}^n [\gamma:\ap_i]$.
Then  

\textbullet\ \  $\min(\Delta^+)=\Pi$ and $\max(\Delta^+)=\{\theta\}$, where $\theta$ is the highest root;

\textbullet\ \ $\Delta^+$ is graded, with rank function $\gamma\mapsto \hot(\gamma)$. Recall that 
$1\le \hot(\gamma)\le h-1$;

\textbullet\ \ $\# \AND=\displaystyle\prod_{i=1}^n \frac{h+m_i+1}{m_i+1}$, where $m_1,\dots,m_n$ are the exponents of $W$, see \cite{cp2}.

\textbullet\ \ if $\Delta$ is of type $\GR{A}{n}$, then $\# \AND$ is the $(n+1)$-th {\it Catalan number\/} 
and the coefficients of $\eus N_{\Delta^+}(t)$ are the {\it Narayana numbers}.
No general uniform expression for  $\eus N_{\Delta^+}(t)$ is known; but case-by-case computations 
show that $\eus N_{\Delta^+}(t)$ is always palindromic and unimodal.

It was conjectured in \cite[Conjecture\,2.1]{oper-X} that $\fX_{\Delta^+}$ satisfies the following properties:
\begin{itemize}
\item[\sf (i)] \ $(\fX_{\Delta^+})^h$ is the permutation on $\AND$  induced by $-w_0$, where $w_0\in W$ 
is the longest element (=\, unique element taking $\Delta^+$  to $-\Delta^+$). In particular, 
$\ord(\fX_{\Delta^+})\in \{h,2h\}$ and $\ord(\fX_{\Delta^+})=h$ if and only if $w_0=-1$;
\item[\sf (ii)] \ Let $\co$  be an $\fX_{\Delta^+}$-orbit in $\AND$. Then the average value 
$\bigl(\sum_{\Gamma\in\co}  \#\Gamma\bigr)/\#\co$ does not depend on $\co$ and equals $\#(\Delta^+)/h=n/2$.
\end{itemize}
This conjecture has been proved in \cite{ast13}. However, similar conjectures in \cite{oper-X}
for several related graded posets (e.g. for $\eus P=\Delta^+\setminus\Pi$) are still open.

We are going to describe a natural class of posets having the similar properties. As there is a simple Lie 
algebra behind $\Delta$, it is natural to explore posets related to representations of semisimple 
(reductive) Lie algebras. Let $\eus P(V)$ be the set of weights of an irreducible 
finite-dimensional representation $V$ of a reductive Lie algebra. Then $\eus P(V)$ is a graded poset. 
We show that  $\eus P(V)\simeq \eus P(V^*)$ and  $\fX_{\eus P(V)}$ can naturally be written as a 
product of two involutions, see Remark~\ref{rem:two-inv}. Another promising observation is that if $V$ is 
{\it weight multiplicity free\/} (=\,\textsf{wmf}), then $\eus P(V)$ is rank symmetric, rank unimodal, and  
Sperner (see precise definitions in Section~\ref{sec:general}). 
However, the class of all (or even \textsf{wmf}) weight posets is too large for interesting properties, and 
we stick to \textsf{wmf} representations associated with $\BZ$-gradings of a simple Lie algebra $\g$.
Given a $\BZ$-grading $\g=\bigoplus_{i\in\BZ}\g(i)$, we are interested in the weight poset, $\Delta(1)$,
of the $\g(0)$-module $\g(1)$. Assuming that $\Delta$ is the root system of $\g$, we may regard 
$\Delta(1)$ as a subposet of $\Delta^+$. Our aim is to demonstrate
that the posets of the form $\Delta(1)$ exhibit many good properties that are akin
to the above properties of $\Delta^+$. The basic material on gradings of $\g$ is gathered in Section~\ref{subs:gradings}.

We consider in details two simplest classes of $\BZ$-gradings:

\textbullet \ \ the {\it abelian\/} gradings,\  $\g=\g(-1)\oplus\g(0)\oplus\g(1)$ (Section~\ref{sect:ab-case});

\textbullet \ \ the {\it extra-special\/} gradings,\ $\g=\bigoplus_{j=-2}^2\g(j)$ 
\ \& \ $\dim\g(2)=1$ (Section~\ref{sect:ext-case}).
\\
To a great extent, our results for them are similar, but the proofs become more tricky in the extra-special 
case.  Using the Kostant-Macdonald identity~\cite[Cor.\,2.5]{macd72},
we prove that $\mdt=\prod_{\gamma\in\Delta(1)}\frac{1-t^{\hot(\gamma)+1}}{1-t^{\hot(\gamma)}}$
(Theorems~\ref{thm:M-polinom-ab} and~\ref{thm:mdt-extra}); in particular, 
$\#\anod=\prod_{\gamma\in\Delta(1)} \frac{\hot(\gamma)+1}{\hot(\gamma)}$. 
All the respective $\eus N$-polynomials are computed, too.
A relationship between $\anod$  and certain elements of the Weyl group $W$ of $\g$ is established
(Theorems~\ref{thm:ab-case} and~\ref{thm:extra-case}). We also provide a model of the poset 
$\anod$ related to the weight poset of a certain representation of the dual Lie algebra $\g^\vee$, see 
Theorems~\ref{thm:isom-poset-ab} and~\ref{thm:isom-poset-extra}. 
Yet certain nice features of the extra-special case have no `abelian' analogues. For instance, we have
$\#\anod=\#\Pi_l{\cdot}(h-1)$, where $\Pi_l$ is the set of long simple roots; $\mdt$ is closely related 
to the Lusztig $t$-analogue of the zero weight multiplicity for the representation of $\g^\vee$ with highest 
weight $\theta^\vee$, see 
Remark~\ref{rem:yet-another}; $\deg\ndt\le 3$
(Theorem~\ref{thm:ndt-extra}) and there is a nice explicit formula for $\ndt$ if $\g$ is of type {\bf ADE} 
(Corollary~\ref{cl:extra-ndt-ADE}).  

Section~\ref{sect:conj} contains numerous examples and our general conjectures on $\anod$ and 
$\fX_{\Delta(1)}$. In particular, we conjecture that 
\par(1) \ our formula for $\mdt$, which us proved in the abelian and extra-special cases, remains valid for {\bf all} $\BZ$-gradings of $\g$;
\par(2) \ if $\g(1)$ is a simple $\g(0)$-module, then  $\ord(\fX_{\Delta(1)})$ equals 
$(\max_{\gamma\in\Delta(1)}\hot(\gamma))+1$;
\par(3) \ if $\co$  is an $\fX_{\Delta(1)}$-orbit in $\anod$, then the average value 
$\bigl(\sum_{\Gamma\in\co}  \#\Gamma\bigr)/\#\co$ does not depend on $\co$ and equals
$\#\Delta(1)/\ord(\fX_{\Delta(1)}$;
\par(4) \ if $\co$  is an $\fX_{\Delta(1)}$-orbit in $\anod$, then the average value 
$\bigl(\sum_{\Gamma\in\co}  \#I(\Gamma)\bigr)/\#\co$ does not depend on $\co$ and equals
$\#\Delta(1)/2$.
\par(5) \ the polynomial $\ndt$ is palindromic if and only if $\Delta(1)$ has a unique rank level of 
maximal size.
\\ 
Our examples show that properties (1)--(3),(5) fail for some \textsf{wmf} representations that 
are not related to $\BZ$-gradings. We also touch upon some aspects of the ``$t=-1$ 
phenomenon"~\cite{stembr}
related to the posets $\Delta(1)$ and their $\eus M$-polynomials. 

In a subsequent  article, we develop some general theory related to the weight posets 
$\Delta(1)$ and discuss manifestations of the cyclic sieving phenomenon \cite{sagan}
in this setting.

\section{Posets, weight posets and gradings}    
\label{sec:general}

\noindent 
We begin with recalling some notation and standard facts on posets, see \cite[Ch.\,3]{stanley}.
The {\it Hasse diagram\/} of $\eus P$ is the directed graph $\eus H(\eus P)$ whose vertex set is 
$\eus P$ and the set of edges is $\{(x,x')\in \eus P\times \eus P\mid \text{ $x$ covers $x'$ }\}$. (Such
an edge is depicted by $\stackrel{x'}{\bullet}\longrightarrow \stackrel{x}{\bullet}$.)
Then 
$\eus P$ is the {\it disjoint union\/} of subposets $\eus P_1$ and $\eus P_2$ 
(denoted $\eus P=\eus P_1\sqcup\eus P_2$), if
$\eus H(\eus P)$ is the disjoint union of graphs $\eus H(\eus P_1)$ and $\eus H(\eus P_2)$.
A poset is said to be {\it connected\/} if it is not a disjoint union of two nonempty subposets.
The following easy observation reduces many problems that are of interest for us to the case of
connected posets.

\begin{lm}    \label{lem:disjoint-union}
If $\eus P=\eus P_1\sqcup\eus P_2$, then 
$\eus M_{\eus P}(t)=\eus M_{\eus P_1}(t){\cdot}\eus M_{\eus P_2}(t)$,
$\eus N_{\eus P}(t)=\eus N_{\eus P_1}(t){\cdot}\eus N_{\eus P_2}(t)$,
and $\ord(\fX_\eus P)=\mathsf{l.c.m.}\{\ord(\fX_{\eus P_1}), \ord(\fX_{\eus P_2})\}$.
\end{lm}

A poset $\eus P$ is {\it graded\/} if it admits a rank function. A {\it rank function\/} on $\eus P$ is a map 
$r:\eus P\to \BN$ such that $r(x)=r(y)+1$ whenever $x$ covers $y$.

As is well known, $\anp$ carries a natural poset structure for any $\eus P$. The quickest way to
introduce it is to use the inclusion of the corresponding upper or lower ideals of $\eus P$.
For $\Gamma,\Gamma'\in\anp$, we set $\Gamma \le_{up} \Gamma'$ if $I(\Gamma)\subset I(\Gamma')$. The similar use of lower ideals yields the opposite poset structure in $\anp$. That is, letting $\Gamma\le_{lo} \Gamma'$ if $I_-(\Gamma)\subset I_-(\Gamma')$, we obtain \\[.5ex]
\centerline{$\Gamma\le_{up}  \Gamma'$ \ if and only if \ 
$\Gamma'\le_{lo}  \Gamma$.}
\\[.6ex]
Sometimes it is convenient to have a separate notation for the (po)sets of upper and lower ideals.
Let $(\eus J_+(\eus P), \subseteq)$ (resp. $(\eus J_-(\eus P), \subseteq)$) be the poset of upper (resp. lower) ideals in $\eus P$ with respect to the usual inclusion. The output of the above discussion is that
\begin{gather*}
(\eus J_+(\eus P), \subseteq) \simeq (\anp, \le_{up}), \quad (\eus J_-(\eus P), \subseteq) \simeq (\anp, \le_{lo}), \\  \text{ and } \ 
(\anp, \le_{lo})  \simeq (\anp, \le_{up})^{opp} .
\end{gather*}
The poset $(\anp, \le_{up})$ is graded, with rank function $\Gamma\mapsto \#I(\Gamma)$. Thus, 
$\eus M_{\eus P}(t)$ is the {\it rank-generating function\/} for $(\anp, \le_{up})$.
Note that 
$\Gamma'$ covers $\Gamma$ w.r.t. $\le_{up}$ if and only if $I(\Gamma')=I(\Gamma)\cup\{x\}$ for some $x\in\eus P$; moreover, 
$I(\Gamma)\cup\{x\}$ is an upper ideal for $x\in \eus P\setminus I(\Gamma)$ 
if and only if $x\in\max(\eus P\setminus I(\Gamma))$. 

\subsection{Weight posets and weight multiplicity free representations}  
\label{subs:wmf}
Let $\h$ be a complex reductive algebraic Lie algebra with $\rk[\h,\h]=m$.
Fix a triangular decomposition $\h=\n^-\oplus\te\oplus\n^+$. 
The root system $\Delta_\h =\Delta(\h,\te)$ is of rank $m$; it is reduced but is not 
necessarily irreducible. Then 
$\Delta^+_\h$ is the set of  roots in $\n^+$ and $\Pi_\h$ is the set of simple roots in $\Delta^+_\h$.
Write $\mathcal X_+$ for the set of dominant weights associated with $\Delta^+_\h$ and 
$W_\h$ for the Weyl group of $\Delta_\h$.

For a simple $\h$-module $\sfr(\lb)$ with highest weight $\lb\in\mathcal X_+$, let 
$\eus P(\sfr(\lb))$ (or merely $\eus P(\lb)$) be
the set of $\te$-weights in $\sfr(\lb)$. Whenever we wish to stress that $\eus P(\lb)$ is associated with $\h$-module, we write $\eus P(\h,\lb)$ for it.
The partial order `$\curle$' in $\eus P(\lb)$ is defined by the 
requirement that, for $\gamma,\nu\in \eus P(\lb)$, $\gamma$ covers $\mu$ if and only if 
$\gamma-\mu\in\Pi_\h $. Hence $\mu\curle\gamma$ if and only if $\gamma-\mu$ is a nonnegative 
integer linear combination of simple roots.
Then $\max\eus P(\lb)=\{\lb\}$.
If $\lb^*$ is the highest weight of the dual representation (i.e., $\sfr(\lb)^*=\sfr(\lb^*)$), then 
$-\lb^*$ is the lowest weight of $\sfr(\lb)$ and $\min \eus P(\lb)=\{-\lb^*\}$.
Let $p$ be any linear function on $\te^*$ such that $p(\ap)=1$ for all
$\ap\in \Pi_\h $. Then all maximal chains from $-\lb^*$ to $\lb$ are of length $p(\lb+\lb^*)$.
Hence $\eus P(\lb)$ is a graded poset. For weight posets, it is convenient to assume that the minimal element
of $\eus P(\lb)$, $-\lb^*$,  has rank one.
The corresponding rank function $r$ is said to be {\it tuned}. It is given by
$\mu\in \eus P(\lb)\stackrel{r}{\mapsto} p(\mu+\lb^*)+1$.

If $V$ is any finite-dimensional $\h$-module, then  $\eus P(V^*)=-\eus P(V)$, hence the posets
$\eus P(V^*)$ and $\eus P(V)$ are anti-isomorphic. But we actually have more.
\begin{lm}   \label{lem:M-polinom-P(V)}
{\rm (i)} The posets $\eus P(V^*)$ and $\eus P(V)$ are naturally isomorphic and {\rm (ii)} the polynomial $ \eus M_{\eus P(V)}(t)$ is palindromic, of degree $\#\eus P(V)$. 
\end{lm}
\begin{proof}
(i) Let $w_0\in W_\h$ be the longest element. Then $-w_0(\Pi_\h)=\Pi_\h$ and $-w_0(\eus P(V))=
\eus P(V^*)$.  Therefore, $\nu\curle \mu$ in $\eus P(V)$ if and only if $-w_0(\nu)\curle -w_0(\mu)$ in
$\eus P(V^*)$. 
\par
(ii)  Suppose that $I\in\eus J_+(\eus P(V))$.  
Since $w_0(\Delta_\h^+)=-\Delta_\h^+$, we have
$w_0(I)\in\eus J_-(\eus P(V))$. Therefore
$I^*:=\eus P(V)\setminus w_0(I)$ is  an upper ideal of complementary cardinality. 
\end{proof}

\begin{rmk}   \label{rem:two-inv}
The reverse operator $\fX_\eus P$ is an element of the symmetric group on $\anp$. Therefore, it 
can be presented as a product of two involutions. An interesting (perhaps, useful) feature of the 
weight posets is that $\fX_{\eus P(V)}$ can be written as such a product in  a very explicit way.
In Lemma~\ref{lem:M-polinom-P(V)}, we considered, for any $I\in \eus J_+(\eus P(V))$, the upper ideal
$I^*=\eus P(V)\setminus w_0(I)$. Then $I^{**}=I$ and
this provides our first ingredient, the involution `$\ast$':
\[
    \Gamma=\min(I(\Gamma)) \mapsto \Gamma^*=\min(I(\Gamma)^*) .
\]
Using the fact that $w_0(\min(I))=\max(w_0(I))$ for any $I\in \eus J_+(\eus P(V))$, we obtain 
\[
    \fX_{\eus P(V)}(\Gamma)=\max(\eus P(V)\setminus I(\Gamma))=
    w_0\bigl( \min(\eus P(V)\setminus w_0(I(\Gamma)))\bigr)=w_0(\Gamma^*) .
\]
Thus, $\fX_{\eus P(V)}$ is the product of involutions '$\ast$' and $w_0$.
\end{rmk}

We say that $V$ is  {\it weight multiplicity free\/} (\textsf{wmf} for short) if every  $\te$-weight space 
of $V$ is one-dimensional. If $V=\bigoplus_i \sfr(\lb_i)$ is \textsf{wmf}, then so is each $\sfr(\lb_i)$ (but 
not vice versa!). Then $\#\eus P(V)=\dim V$, $\eus P(V)$ is the disjoint union of the posets 
$\eus P(\lb_i)$, and there is a one-to-one correspondence between $\eus J_+(\eus P(V))$ and
the $(\te\oplus\n^+)$-stable subspaces of $V$.
\par
The list of the irreducible \textsf{wmf} representations of simple Lie algebras is obtained by 
R.\,Howe~\cite[4.6]{howe}, see also \cite[Table\,1]{mmj}. Taking tensor products, one derives from it 
the corresponding list for the semisimple Lie algebras. The following is obvious.

\begin{lm}  \label{lem:tensor-prod-dir-sum}
Let  $(\h_i,\sfr(\lb_i))$, $i=1,\dots,l$, be irreducible representations of reductive Lie algebras. Then
\par 
{\sf \bfseries (1)} the representation $\sfr=\sfr(\lb_1)\otimes\dots\otimes\sfr(\lb_l)$ of\/ 
$\h=\h_1\times\ldots\times\h_l$ is irreducible, with highest weight $\lb_1+\ldots+\lb_l$, and
\[  
   \eus P\bigl(\sfr(\lb_1)\otimes\dots\otimes\sfr(\lb_l)\bigr)=\eus P(\lb_1+\ldots+\lb_l)=\eus P(\lb_1)\times\ldots
   \times\eus P(\lb_l) ,
\]
the direct product of the weight posets $\eus P(\lb_i)$.
\par 
{\sf \bfseries (2)} If each $(\h_i,\sfr(\lb_i))$ is \textsf{wmf}, then so is $(\h,\sfr)$.
\end{lm}
\begin{ex}   \label{ex:sl_n-simplest}
The root system of  $\h=\sln$  is of type $\GR{A}{n-1}$, and we use the common notation for roots, etc.,
see Tables in ~\cite{bour, t41}. 
That is, $\ap_i=\esi_i-\esi_{i+1}$ are the simple roots and $\varpi_i=\esi_1+\ldots+\esi_i$
are the fundamental weights, $i=1,\dots,n-1$. The standard $n$-dimensional representation 
of $\sln$ is \textsf{wmf} and its weights are $\{\esi_i,\ i=1,\dots,n\}$. Here $\esi_1=\varpi_1$ is the 
highest weight and  $\eus P(\varpi_1)$ is the $n$-element  chain $\eus C_{n}$:
\begin{center}   $\eus H(\eus P(\varpi_1))$: \qquad
\begin{picture}(110,22)(0,3)
\multiput(10,8)(20,0){2}{\color{MIXT}\circle*{4}}
\multiput(70,8)(20,0){2}{\color{MIXT}\circle*{4}}

\multiput(11,8)(60,0){2}{\vector(1,0){18}}

\put(45,5){$\cdots$}
\put(-5,5){{\small $\esi_n$}}
\put(95,5){{\small $\esi_1$}}
\end{picture}
\end{center}
\vskip.5ex\noindent
If $\h=\sln\times\mathfrak{sl}_m$ and $V=\sfr(\varpi_1)\otimes\sfr(\varpi'_1)$ is the tensor product of 
the standard representations, then 
$\eus P(V)=\eus P(\varpi_1+\varpi'_1)\simeq \eus C_n\times\eus C_m$. The Hasse diagram 
$\eus H(\eus C_n\times\eus C_m)$ is the rectangle of size $m\times n$. Clearly, all posets
$\eus C_{n_1}\times\dots\times \eus C_{n_l}$ are associated with \textsf{wmf} representations, but one shouldn't expect much from them, if $l\ge 4$, see Section~\ref{sect:conj}.
\end{ex}

For a graded  poset $\eus P$, let  $\eus P_i$ denote the set of elements of rank $i$.  The sets 
$\eus P_i$ are said to be the {\it (rank) levels\/} of $\eus P$.
Suppose that $\eus P=\bigsqcup_{i=1}^d \eus P_i$. 
Then $\eus P$ is {\it rank symmetric\/} if $\#\eus P_i=\#\eus P_{d+1-i}$ for all $i$; 
it is {\it rank unimodal\/} if 
$\#\eus P_1\le \#\eus P_2\le \cdots \le \#\eus P_k\ge \#\eus P_{k+1}\ge\cdots \ge \#\eus P_d$ 
for some $1\le k\le d$.
The poset $\eus P$ is said to be {\it Sperner}, if the largest size of an antichain is
equal to $\max\{\#\eus P_i, \ 1\le i\le d\}$. 

\begin{lm}    \label{lem:sperner-sl2}
For any simple \textsf{wmf}\/ $\h$-module $\sfr(\lb)$, the graded poset $\eus P(\lb)$ is rank symmetric, 
rank unimodal,  and Sperner. 
\end{lm}
\begin{proof}
Consider a {\it principal $\tri$-subalgebra\/} $\ah\subset [\h,\h]$ associated with our fixed triangular decomposition of $\h$. That is, $\ah\simeq\tri$ has a basis $(X,H,Y)$ such that $[X,Y]=H$, $[H,X]=2X$,
$[H,Y]=-2Y$, $H\in\te$ is a regular element, 
and $X$ (resp. $Y$) is a sum of root vectors corresponding to $\Pi_\h $
(resp.  $-\Pi_\h $). Then the representation of $\ah$ on $\sfr(\lb)$ translates into a representation of 
$\ah$ on the graded poset $\eus P(\lb)$, as defined in \cite{pr82}. Therefore, the main result of \cite{pr82} yields all the assertions.
\end{proof}

\noindent It follows from \cite{pr82} that the poset $\eus P(\lb)$ is also {\it strongly Sperner}, but we do need this here.

\begin{rmk}   \label{rem:dynkin}
The idea to use $\tri$-subalgebras for obtaining  properties of irreducible representations
goes back to E.B.\,Dynkin. In 1950, he invented principal $\tri$-subalgebras of semisimple Lie algebras
and proved, using them, that a certain partition of $\sfr(\lb)$ into layers, 
$\sfr(\lb)=\bigoplus_{j=1}^N \sfr(\lb)_j$, yields a symmetric unimodal sequence
$\{d_j=\dim \sfr(\lb)_j\}_{j=1}^N$, see \cite[Theorem\,4]{dy50}. In the special case of {\textsf wmf} 
representations, one has $d_j=\#\eus P(\lb)_j$ and Dynkin's result translates into the assertion that  
$\eus P(\lb)$ is rank symmetric and rank unimodal.
\end{rmk}

\subsection{Gradings of simple Lie algebras} 
\label{subs:gradings}
Interesting \textsf{wmf}~representations and weight posets are associated with gradings of simple Lie 
algebras. We refer to \cite[Ch.\,3,\,\S\,3]{t41} for generalities on gradings. In what follows, $\g$ is a 
complex simple Lie algebra with a fixed triangular decomposition $\g=\ut^-\oplus\te\oplus\ut^+$. 
The associated root system is $\Delta=\Delta(\g,\te)\subset\te^*$, and we use all the relevant 
notation on $\Delta$ presented in 
Introduction. Additionally, $\te^*_{\BR}$ is the $\BR$-span of $\Delta$ in $\te^*$ and $(\ ,\ )$ is a $W$-invariant scalar product in $\te^*_{\BR}$. For $\gamma\in\Delta$, we set 
$\gamma^\vee=2\gamma/(\gamma,\gamma)$. Then
$\Delta^\vee=\{\gamma^\vee \mid \gamma\in\Delta\}$ is the {\it dual root system}, 
$(\Delta^\vee)^+=(\Delta^+)^\vee$ is a set of positive roots in $\Delta^\vee$, and 
$\Pi^\vee=\{\ap^\vee\mid \ap\in \Pi\}$ is the set of simple roots in $(\Delta^\vee)^+$. However, 
$\theta^\vee$ appears to be the highest root in $(\Delta^\vee)^+$ is and only if all roots of
$\Delta$ have the same length (i.e., $\Delta\in\{\mathbf{ADE}\}$). Write $s_\gamma$ for the 
{\it reflection\/} in $W$ with respect to $\gamma\in \Delta$. Note that $s_\gamma=s_{\gamma^\vee}$.

\textbullet \quad Let $\g=\bigoplus_{i\in\BZ}\g(i)$ be a $\BZ$-grading. Since any derivation of $\g$ 
is inner,  $\g(i)=\{x\in\g \mid [\tilde h,x]=ix\}$ for some $\tilde h\in \g(0)$. The element 
$\tilde h$ is said to be {\it defining\/} for the grading in question. Since $\tilde h$ is semisimple,
$\g(0)$ is reductive, $\rk\g=\rk\g(0)$, and $\g(i)$ is a \textsf{wmf} $\g(0)$-module for $i\ne 0$. Because 
$\tilde h$ lies in the centre of $\g(0)$, it is convenient to introduce a reductive subalgebra
$\widetilde{\g(0)}$ such that $\g(0)=\widetilde{\g(0)}\oplus\langle \tilde h\rangle$. 
\par
Without loss of generality, one may assume that $\tilde h\in\te$ and $\ap(\tilde h)\ge 0$ for all $\ap\in\Pi$. Then $\te\subset\g(0)$ and the grading is fully determined by the partition $\Pi=\bigsqcup_{i\ge 0} \Pi(i)$, where  $\Pi(i)=\{\ap\in\Pi\mid \ap(\tilde h)=i\}$.
If $\Delta(i)=\{\gamma\in \Delta\mid \gamma(\tilde h)=i\}$, then $\Delta(i)$ is the set of roots of $\g(i)$,  
and $\Delta=\bigsqcup_{i\in\BZ} \Delta(i)$. We repeatedly use the property that the last partition is a
``grading'' of $\Delta$, i.e., if $\gamma\in\Delta(i)$, $\gamma'\in\Delta(j)$, and $\gamma+\gamma'$ is a root, then $\gamma+\gamma'\in\Delta(i+j)$.
\par
In this setting, we automatically obtain a distinguished triangular decomposition:
\[
     \g(0)=(\g(0)\cap\ut^-)\oplus\te \oplus (\g(0)\cap\ut^+) .
\]
Then $\Delta(0)^+:=\Delta^+\cap\Delta(0)$ is the set of roots of $\g(0)\cap\ut^+$, 
$\Pi(0)$ is the set of simple roots in $\Delta(0)^+$, and 
\[
    \Delta^+=\Delta(0)^+\sqcup \Delta(1)\sqcup\Delta(2)\sqcup\dots 
\]
We may (and will) use the notation and results of Section~\ref{subs:wmf} with $\h=\g(0)$, 
$\n^+=\g(0)\cap\ut^+$,
$\Pi_\h =\Pi(0)$, etc. For instance, $\Delta(i)=\eus P(\g(i))$.
The Weyl group of $\Delta(0)$, $W(0)$, is a parabolic subgroup of $W$. Let $W^0$ denote the 
set of representatives of minimal length for the cosets $W/W(0)$. In other words \cite[1.10]{hump},
\beq
         W^0=\{w\in W \mid w(\ap)\in \Delta^+ \quad \forall \ap\in\Delta(0)^+\} .
\eeq
\begin{rmk}   \label{rem:g1-dostat}
Since $\g(i)$ and $\g(-i)$ are dual $\g(0)$-modules,  the posets $\Delta(i)$ and $\Delta(-i)$ are
isomorphic. Furthermore, if one is interested in possible simple $\g(0)$-modules occurring in $\g(i)$ 
with $i> 0$, then it is enough to consider only the simple $\g(0)$-modules in $\g(1)$,
see~\cite[\S\,1.2, \S\,2.1]{vi76}. 
(For $i >1$,  the problem is reduced to considering the induced grading of
a certain simple subalgebra of $\g$.)
For this reason, it suffices to consider defining elements $\tilde h\in\te$ such that $\ap(\tilde h)\in\{0,1\}$, i.e., 
$\Pi=\Pi(0)\sqcup\Pi(1)$.
The corresponding $\BZ$-gradings are said to be {\it standard}.
Moreover, if $\#\Pi(1)=k$, then we call it a $k$-{\it standard\/} grading.
A standard $\BZ$-grading can be represented by the Dynkin diagram of $\g$,
where the vertices in $\Pi(1)$ are marked. If $\Pi(1)=\{\ap_{i_1},\dots,\ap_{i_k}\}$, then the 
$\ap_{i_j}$'s are precisely the lowest weights of the simple $\g(0)$-modules in $\g(1)$, the centre of 
$\g(0)$ is $k$-dimensional, and $\g(1)$ is a direct sum of $k$ simple $\g(0)$-modules. Thus, for the 
standard $\BZ$-gradings, one obtains
\[
 \text{\it $\widetilde{\g(0)}$ is semisimple $\Leftrightarrow$ the grading is $1$-standard $\Leftrightarrow$ $\g(1)$ is a simple $\widetilde{\g(0)}$-module}.
\]
\end{rmk}
\begin{rmk} 
The weight posets $\Delta(i)$, $i>0$, can be visualised as follows. 
If $\gamma'-\gamma=\ap\in\Pi$, then the 
edge connecting $\gamma$ and $\gamma'$ in the Hasse diagram $\eus H(\Delta^+)$ is said to be {\it of type\/} $\ap$. 
Given a standard $\BZ$-grading of $\g$, let us remove from $\eus H(\Delta^+)$ all the edges of types 
in $\Pi(1)$. This yields a disconnected graph. Each connected component of it is the Hasse 
diagram of either the set of positive roots of a simple factor of $\g(0)$ (if it contains roots from $\Pi(0)$) 
or the weight poset of a simple $\g(0)$-module in some $\g(i)$, $i>0$. The set of weights of a simple $\g(0)$-module in some $\g(i)$, $i\ge 1$, is precisely the set of roots $\gamma$ with fixed values
$[\gamma : \ap]$ for all $\ap\in\Pi(1)$, see e.g.~\cite[3.5]{t41}. Therefore, 
\par 
{\bf --} \ each weight poset $\Delta(i)$, $i\ge 1$, is a subposet of 
$\Delta^+$. 
\par 
{\bf --} \ the connected component containing a node $\ap\in\Pi(1)$ is the Hasse diagram of the weight
poset of a simple $\g(0)$-module in $\g(1)$, 
\par 
{\bf --} \  the tuned rank function for the whole poset 
$\Delta(1)$ is  the restriction to $\Delta(1)$ of the usual height function $\hot: \Delta^+\to \BN$.
\end{rmk}
\textbullet \quad 
Let $\g=\bigoplus_{i\in\BZ_m}\g_i$ be a periodic (or $\BZ_m$-) grading. Such gradings are in a 
one-to-one correspondence with the automorphisms $\sigma\in \Aut(\g)$ of order $m$. Here again 
$\g_0$ is reductive and $\g_i$ is a $\g_0$-module. If $\sigma$ is inner, then $\rk\g_0=\rk\g$ 
and $\g_i$ is a \textsf{wmf} $\g_0$-module for $i\ne 0$. For inner automorphisms $\sigma$, one 
may further assume that $\te\subset\g_0$ and obtain the partition $\Delta=\bigsqcup_{i=0}^{m-1}\Delta_i$ with $\Delta_i=\eus P(\g_i)$. A classification of periodic automorphisms of $\g$ is obtained by V.G.\,Kac (1969). In particular, an explicit description of inner periodic automorphisms
of $\g$ and the respective $\g_0$-modules $\g_1$ can be given in terms of the {\sl extended Dynkin diagram\/} of $\g$, see \cite[\S\,8]{vi76} or \cite[Ch.\,3,\,\S 3.7]{t41} for a thorough self-contained description.

Likewise, for periodic gradings, it suffices to consider $\g_0$-modules $\g_1$. We mention the following useful property of periodic gradings:
{\it $\g_1$ is a simple $\g_0$-module if and only if $\g_0$ is semisimple}, see  \cite[Prop.18]{vi76}.
Note, however, that all $\Delta_i$ are {\bf not} usually subsets of $\Delta^+$! Therefore, 
the rank function for $\Delta_1$ has no relation to  $\hot(\ )$. 

Our main interest lies in the study of representations (posets) related to $\BZ$-gradings. In 
Section~\ref{sect:conj}, we use
\textsf{wmf} representations associated with periodic gradings to demonstrate that properties of the 
antichains and reverse operators become worse for them.

\begin{rmk}   \label{rmk:isom-poset}
In this article, we are primarily interested in combinatorial properties of weight posets. It is therefore 
helpful to keep in mind that the weight posets of two \textsf{wmf} representations of {\bf different} 
Lie algebras can be isomorphic. A list of such isomorphisms appears in \cite[Theorem\,2.1]{mmj}.
\end{rmk}
\subsection{Special classes of $\BZ$-gradings}  
\label{subs:classes}

Our goal is to demonstrate that the weight posets of the form $\Delta(1)$ for $\BZ$-gradings of $\g$ 
exhibit the best possible properties related to antichains, their $t$-analogues, and reverse operators. 
We will consider in details the following two cases:

\vskip.7ex 
{\it\bfseries The abelian gradings}: $\g=\g(-1)\bigoplus\g(0)\bigoplus\g(1)$. 
\\[.5ex] Here $\g(0)\bigoplus\g(1)$ 
is a parabolic subalgebra and $\g(1)$ is its (abelian!) nilradical.  In this case  $\g(1)$ is a simple 
$\g(0)$-module and therefore $\widetilde{\g(0)}$ is semisimple and $\Pi(1)=\{\ap_i\}$. The admissible simple roots $\ap_i$ are characterised by 
the property that $[\theta:\ap_i]=1$. Therefore, such gradings exist
if and only if $\g\ne \GR{G}{2}, \GR{F}{4}, \GR{E}{8}$. The following table provides the  admissible
simple roots (with numbering as in \cite[Table\,1]{t41}): 
\vskip1ex
\begin{center}
\begin{tabular}{|cc||cc|}
\hline
$\GR{A}{n}$ & $\ap_1,\dots,\ap_n$ & $\GR{D}{n}, \ n\ge 4$ & $\ap_1,\ap_{n-1},\ap_n$ \\ \hline
$\GR{B}{n}, \ n\ge 2$ & $\ap_1$ & $\GR{E}{6}$ & $\ap_1,\ap_5$ \\ \hline
$\GR{C}{n}, \ n\ge 2$ & $\ap_n$ & $\GR{E}{7}$ & $\ap_1$  \\
\hline
\end{tabular}
\end{center}
In particular, all $1$-standard $\BZ$-gradings for $\GR{A}{n}$ are abelian!
If an abelian grading corresponds to $\ap_i\in\Pi$, then upon the identification of $\te_\BR$ and 
$\te^*_\BR$, the defining element $\tilde h$ appears to be the minuscule fundamental 
weight $\vp_i^\vee$ of the dual root system $\Delta^\vee$. Indeed, the weight $\vp_i^\vee$ satisfies the relation
\[
   (\vp_i^\vee,(\gamma^\vee)^\vee)=(\vp_i^\vee,\gamma)=\gamma(\tilde h)\in \{-1,0,1\} \quad 
   \text{for all } \ 
   \gamma^\vee\in \Delta^\vee .
\]
Hence it is minuscule \cite[Ch.\,VIII,\ \S\,7, n$^0$3]{bour7-8}. Furthermore, 
$(\vp_i^\vee,\ap_i)=1$ and $(\vp_i^\vee,\ap)=0$ for $\ap\in\Pi\setminus \{\ap_i\}$. Hence $\vp_i^\vee$
is fundamental. Note also that
$W(0)$ is the stabiliser of $\vp_i^\vee$ in $W$.
\vskip1.2ex \indent
{\it\bfseries The extra-special gradings}: 
$\g=\g(-2)\bigoplus\g(-1)\bigoplus\g(0)\bigoplus\g(1)\bigoplus\g(2)$ \& $\dim\g(2)=1$. 
\\    
Any simple Lie algebra has a unique, up to conjugation, $\BZ$-grading of this form, and w.l.o.g. we may
assume that $\Delta(2)=\{\theta\}$. Upon the identification of $\te_\BR$ and $\te^*_\BR$, the defining
element $\tilde h$ is recognised as the coroot $\theta^\vee$.
That is, $\Delta(i)=\{\gamma\in\Delta\mid (\gamma,\theta^\vee)=i\}$ and $W(0)$ is the stabiliser of 
$\theta$ (or $\theta^\vee$) in $W$. For $\GR{A}{1}$ only, we have $\Delta(1)=\varnothing$, i.e., 
the grading appears to be not standard. (This case better fits in the setting of abelian 
gradings.) For all other simple Lie algebras, $\theta\not\in \Pi$ and the extra-special grading is standard.
Then $\Pi(1)=\{\ap\in\Pi\mid (\gamma,\theta^\vee)\ne 0\}$. Therefore $\widetilde{\g(0)}$ is semisimple 
and $\g(1)$ is a simple $\g(0)$-module if and only if $\theta$ is a multiple of a fundamental weight, i.e.,
$\g$ is not of type $\GR{A}{n}$.
\par
Note also that $\g(1)$ is a symplectic $\widetilde{\g(0)}$-module  (hence $\dim\g(1)$ is even) 
and $\g(1)\bigoplus\g(2)$ is a Heisenberg Lie algebra.

\section{Antichains and upper ideals in the abelian case}
\label{sect:ab-case}

\noindent
In this section, only abelian $\BZ$-gradings of $\g$ are considered. We give a description of the antichains and upper ideals in the poset $\Delta(1)$ and compute the corresponding $\eus M$- and
$\eus N$-polynomials.

For $w\in W$, let $N(w)\subset \Delta^+$ be the inversion set and $\ell(w)$ the length of $w$. 
Recall that $\ell(w)=\#N(w)$. It is readily seen that both $N(w)$ and 
$\Delta^+\setminus N(w)$ are {\it closed subsets\/} of $\Delta^+$. (That is, if
$\gamma_1,\gamma_2\in N(w)$ and $\gamma_1+\gamma_2$ is a root, then 
$\gamma_1+\gamma_2\in N(w)$; and likewise for $\Delta^+\setminus N(w)$.)
Conversely, if $M\subset\Delta^+$ has the property that both $M$ and $\Delta^+\setminus M$ are 
closed, then $M=N(w)$ for a unique $w\in W$, see e.g. \cite[p.\,663]{papi94}.
 
\begin{thm}    \label{thm:ab-case}
In the abelian case, there is a natural bijection $W^0 \stackrel{1:1}{\longleftrightarrow} \AN(\Delta(1))$ that takes
$w\in W^0$ to $\Gamma_w:=\min(\Delta(1)\setminus N(w))$. Furthermore,
\begin{itemize}
\item[\sf (i)] \ $I_w=\Delta(1)\setminus N(w)$ is an upper ideal in $\Delta(1)$;
\item[\sf (ii)] \  if $\gamma\in I_w$, then $\gamma\in\min(I_w)$ if and only if $w(\gamma)\in\Pi$; in other words,
$\Gamma_w=\min(I_w)=w^{-1}(\Pi)\cap \Delta(1)$;
\item[\sf (iii)] \  if $\gamma\in N(w)$, then $\gamma\in\max(N(w))$ if and only if $w(\gamma)\in-\Pi$; in other words,\\
$\max(N(w))=\max(\Delta(1)\setminus I_w)=-w^{-1}(\Pi)\cap \Delta(1)$.
\end{itemize}
In particular, $\#\AN(\Delta(1))=\#(W/W(0))$ \ and \ $\# I_w=\dim\g(1)-\ell(w)$.
\end{thm}
\begin{proof} We have $\Delta^+=\Delta(0)^+\cup\Delta(1)$. If $w\in W^0$, then $N(w)\subset \Delta(1)$
and we set
\[
    I_w:=\Delta(1)\setminus N(w)=\{\gamma\in \Delta(1) \mid w(\gamma)>0\} .
\]

1) If $\delta\in\Delta(0)^+$, $\gamma\in I_w$, and $\delta+\gamma$ is a root, then $\delta+\gamma
\in I_w$, since $\Delta^+\setminus N(w)$ is closed. Hence $I_w$ is an upper ideal of the poset 
$\Delta(1)$. Conversely, it is immediate that if $I\subset \Delta(1)$ is an upper ideal, then both 
$\Delta(1)\setminus I$ and 
$I\cup\Delta(0)^+$ are closed. Hence $\Delta(1)\setminus I=N(w)$ for a unique $w\in W^0$.
This yields the desired bijection and (i).

2) Suppose that $\gamma\in I_w$ and $\gamma\not\in \min(I_w)$. Then $\gamma=\delta+\gamma'$ for
some $\delta\in \Pi(0)$ and $\gamma'\in I_w$.  Hence $w(\gamma)=w(\delta)+w(\gamma')$ is a sum of
positive roots, i.e., $w(\gamma)\not\in\Pi$. 
\par Conversely,  
assume that $w(\gamma)\not\in\Pi$, i.e., 
$w(\gamma)=\mu_1+\mu_2$ with $\mu_i\in \Delta^+$. Then 
\[
w^{-1}(\mu_1)+w^{-1}(\mu_2)=\gamma\in \Delta(1) .
\]
We may assume that $w^{-1}(\mu_1)\in \Delta(0)$ and $w^{-1}(\mu_2)\in\Delta(1)$. Then
$w^{-1}(\mu_2)\in I_w$. The root $w^{-1}(\mu_1)$
is either positive or negative.  The assumption $w^{-1}(\mu_1)\in -\Delta(0)^+$ contradicts the fact that 
$w\in W^0$, and if  $w^{-1}(\mu_1)\in \Delta(0)^+$,  then $\gamma-w^{-1}(\mu_1)\in I_w$, hence
$\gamma\not\in \min(I_w)$. This proves (ii).

3) The proof of (iii) is similar to the previous argument and left to the reader. 
\end{proof}

\begin{cl}  \label{cor:summa}
For $w\in W^0$,  we have \ $\#\max(N(w))+ \#\Gamma_w\le \rk\g$.
\end{cl}
\begin{proof}
We have $w(\Gamma_w)\subset \Pi$, \ $w(-\max(N(w)))\subset \Pi$, and all these 
roots are different.
\end{proof}

\begin{rem}   \label{rem:partition-Pi}
Since $\Gamma_w$ is an antichain and $\Delta(2)=\varnothing$, the roots in $\Gamma_w$ are
pairwise {\it strongly orthogonal}. That is, $\gamma\pm\gamma'$ is not a root for all
$\gamma_,\gamma'\in\Gamma_w$. Therefore, $w(\Gamma_w)$ is a strongly orthogonal set of simple 
roots, i.e., $w(\Gamma_w)$ represents a totally disjoint subset of the Dynkin diagram. The same 
also holds for $\max(N(w))$ and $w(-\max(N(w)))\subset \Pi$. Hence $w(\Gamma_w)$ and 
$w(-\max(N(w)))$ are totally disjoint subsets of $\Pi$ without common elements.
One can verify that the equality 
\[
     \#\Gamma_w+ \#\max(N(w))= \rk\g, \ \text{ i.e., } \ w(\Gamma_w) \bigsqcup w(-\max(N(w)))=\Pi
\] 
occurs for some $w\in W^0$ if and only if $\Delta$ is of type $\GR{A}{n}$ or
$\GR{C}{n}$. (And since the Dynkin diagram is a tree, such a partition of $\Pi$ into two totally disjoint subsets is unique, up to permutation). For $\GR{C}{n}$, one considers the unique abelian $\BZ$-grading with
$\Delta(1)=\{\esi_i+\esi_j \mid 1\le i\le j\le n\}$. Then one of the possibilities is
$\Gamma_w=\{\esi_1+\esi_n,\esi_2+\esi_{n-1},\dots\}$ and $\max(N(w))=\{
\esi_2+\esi_n, \esi_3+\esi_{n-1},\dots \}$. Here $\# \Gamma_w=\left[\frac{n+1}{2}\right]$ and 
$\# \max(N(w))=\left[\frac{n}{2}\right]$.
For $\GR{A}{n}$, one has to take the 
abelian grading corresponding to $\ap_i$ with $i=[\frac{n+1}{2}]$ or $i=n+1-[\frac{n+1}{2}]$. (Hence 
there are two possible gradings for  $\GR{A}{2k}$ and only one grading for $\GR{A}{2k-1}$.) The details
are left to the reader.
\end{rem}
For any subset $S\subset W$, we define its Poincar\'e polynomial by
$\boldsymbol{S}(t)=\sum_{w\in S}t^{\ell(w)}$.

\begin{thm}   \label{thm:M-polinom-ab}
For the abelian gradings, we have \ $\displaystyle
   \eus M_{\Delta(1)}(t)=\prod_{\gamma\in\Delta(1)}\frac{1-t^{\hot(\gamma)+1}}{1-t^{\hot(\gamma)}}$. 
\end{thm}
\begin{proof}
By the Kostant-Macdonald identity \cite[Cor.\,2.5]{macd72}, we have
\[ 
   \bow(t)= \prod_{\gamma\in\Delta^+}\frac{1-t^{\hot(\gamma)+1}}{1-t^{\hot(\gamma)}}
  \  \text{ and } \
     \boldsymbol{W(0)}(t)=\prod_{\gamma\in\Delta(0)^+}\frac{1-t^{\hot(\gamma)+1}}{1-t^{\hot(\gamma)}} .
\] 
The bijection $W^0\times W(0) \stackrel{\sim}{\longrightarrow} W$ has the property that 
if $w_1\in W^0$ and $w_2\in W(0)$, then $\ell(w_1w_2)=\ell(w_1)+\ell(w_2)$, see \cite[1.10]{hump}.
It follows that $\displaystyle 
\prod_{\gamma\in\Delta(1)}\frac{1-t^{\hot(\gamma)+1}}{1-t^{\hot(\gamma)}}=\boldsymbol{W^0}(t)$
and 
\[
    [t^i]\prod_{\gamma\in\Delta(1)}\frac{1-t^{\hot(\gamma)+1}}{1-t^{\hot(\gamma)}}=\#\{w\in W^0\mid 
    \ell(w)=i\} .
\]
By Theorem~\ref{thm:ab-case},  the latter equals the number of upper ideals of cardinality 
$\#\Delta(1)-i$. Since $\eus M_{\Delta(1)}(t)$ is palindromic and of degree $\#\Delta(1)$ 
(Lemma~\ref{lem:M-polinom-P(V)}), we are done.
\end{proof}

\begin{cl}   \label{cor:card-ab}
For the abelian gradings, we have 
$\#\anod=\eus M_{\Delta(1)}(1)=\prod_{\gamma\in\Delta(1)}\frac{\hot(\gamma)+1}{\hot(\gamma)}$.
\end{cl}
We are not aware of a uniform general expression for the $\eus N$-polynomials related to the abelian 
case, but it is not hard to compute them directly. The resulting list is provided below. For each item, 
we point out the corresponding simple root $\ap_i$ of $\g$, the semisimple
Lie algebra $\widetilde{\g(0)}$, and the highest weight $\psi$ of the irreducible representation of
$\widetilde{\g(0)}$ in $\g(1)$.
Write $\varpi_i$ for the $i$-th fundamental weight of a simple factor of $\widetilde{\g(0)}$.

\vskip1.5ex 
\centerline{
{\it\bfseries The abelian $\BZ$-gradings}:  \ 
$(\g=\bigoplus_{j=-1}^1 \g(j),\ap_i)\leadsto (\widetilde{\g(0)},\psi)$  
{\it\bfseries and $\eus N$-polynomials}:}

\vskip1ex\noindent 
{\it 1.} $(\GR{A}{n+m-1},\ap_m){\leadsto} (\GR{A}{n-1}{\times}\GR{A}{m-1}, \varpi_1{+}\varpi_1')$, 
\ $\dim\sfr(\varpi_1{+}\varpi_1'){=}nm$, \ 
$\#\AN(\eus P(\varpi_1{+}\varpi_1'))=\genfrac{(}{)}{0pt}{}{n+m}{m}$.

\vskip1ex
$\eus N_{\Delta(1)}(t)=\displaystyle\sum_{i\ge 0} \genfrac{(}{)}{0pt}{}{n}{i}\genfrac{(}{)}{0pt}{}{m}{i}t^i$,
\ $\eus N'_{\Delta(1)}(1)=\displaystyle\frac{(n+m-1)!}{(n-1)!\,(m-1)!}$, \
$\eus N'_{\Delta(1)}(1)/\eus N_{\Delta(1)}(1)=\displaystyle\frac{n{\cdot} m}{n+m}$.

\vskip1.5ex\noindent 
{\it 2.} $(\GR{B}{n},\ap_1)\leadsto (\GR{B}{n-1}, \varpi_1)$, \quad $\dim\sfr(\varpi_1)=2n-1$, \quad
$\#\AN(\eus P(\varpi_1))=2n$.

$\eus N_{\Delta(1)}(t)=1+(2n-1)t$,
\quad
$\eus N'_{\Delta(1)}(1)/\eus N_{\Delta(1)}(1)=\displaystyle\frac{2n-1}{2n}$.

\vskip1.5ex\noindent 
{\it 3.} $(\GR{C}{n},\ap_n) \leadsto (\GR{A}{n-1}, 2\varpi_1)$, \quad $\dim\sfr(2\varpi_1)=n(n+1)/2$, \ 
$\#\AN(\eus P(2\varpi_1))=2^n$.

$\eus N_{\Delta(1)}(t)=\displaystyle\sum_{i\ge 0} \genfrac{(}{)}{0pt}{}{n+1}{2i}t^i$,
\quad
$\eus N'_{\Delta(1)}(1)/\eus N_{\Delta(1)}(1)=\displaystyle\frac{n+1}{4}$.

\vskip1ex\noindent 
{\it 4.} $(\GR{D}{n},\ap_{n-1} \text{ or } \ap_n)\leadsto (\GR{A}{n-1}, \varpi_2)$, \quad $\dim\sfr(\varpi_2)=n(n-1)/2$, \ 
$\#\AN(\eus P(\varpi_2))=2^{n-1}$.

$\eus N_{\Delta(1)}(t)=\displaystyle\sum_{i\ge 0} \genfrac{(}{)}{0pt}{}{n}{2i}t^i$,
\quad
$\eus N'_{\Delta(1)}(1)/\eus N_{\Delta(1)}(1)=\displaystyle\frac{n}{4}$.

\vskip1ex\noindent 
{\it 5.} $(\GR{D}{n},\ap_1)\leadsto (\GR{D}{n-1}, \varpi_1)$, \quad $\dim\sfr(\varpi_1)=2n-2$, \quad
$\#\AN(\eus P(\varpi_1))=2n$.  \nopagebreak

$\eus N_{\Delta(1)}(t)=1+(2n-4)t+t^2$,
\quad
$\eus N'_{\Delta(1)}(1)/\eus N_{\Delta(1)}(1)=1$.

\vskip1ex\noindent 
{\it 6.} $(\GR{E}{6}, \ap_{1} \text{ or } \ap_5)\leadsto (\GR{D}{5}, \varpi_4 \text{ or } \varpi_5)$, \quad $\dim\sfr(\varpi_5)=16$,
\quad 
$\#\AN(\eus P(\varpi_5))=\frac{9{\cdot}12}{1{\cdot}4}=27$.

$\eus N_{\Delta(1)}(t)=1+16t+10t^2$,
\quad
$\eus N'_{\Delta(1)}(1)/\eus N_{\Delta(1)}(1)=16/12$.

\vskip1ex\noindent 
{\it 7.} $(\GR{E}{7},\ap_1)\leadsto (\GR{E}{6}, \varpi_1)$, \quad $\dim\sfr(\varpi_1)=27$,
\quad 
$\#\AN(\eus P(\varpi_1))=\frac{10{\cdot}14{\cdot}18}{1{\cdot}5{\cdot}9}=56$.

$\eus N_{\Delta(1)}(t)=1+27t+27t^2+t^3$,
\quad
$\eus N'_{\Delta(1)}(1)/\eus N_{\Delta(1)}(1)=27/18$.
\vskip1ex\noindent
{\sl Comments on computations:}
\par \textbullet \ \ For item~1, $\Delta(1)$ is the direct product of two chains, see 
Example~\ref{ex:sl_n-simplest}. Here $\eus H(\Delta(1))=\eus H(\eus P(\varpi_1{+}\varpi_1'))$ is a rectangle and computations are straightforward. 
\par \textbullet \ \ For item~2, $\Delta(1)=\eus P(\GR{B}{n-1},\varpi_1)\simeq \eus C_{2n-1}$;
\par \textbullet \ \ For item~3,  $\g(1)$ is the unique maximal 
abelian $\be$-ideal in $\ut$, and $\eus N_{\Delta(1)}(t)$ is the {\it upper covering polynomial\/}
for the poset of all abelian ideals, see \cite[Theorem\,6.2]{coveri}. 
\par \textbullet \ \ Item~4 is related to item~3 via a shift of rank of $\widetilde{\g(0)}$, because 
there is an isomorphism of weight posets  
$\eus P(\GR{A}{n-1}, 2\varpi_1)\simeq \eus P(\GR{A}{n}, \varpi_2)$, see \cite[Theorem\,2.1]{mmj}.
\par \textbullet \ \ For items~5 and 6, the maximal rank level of $\Delta(1)$ is of size $2$. Hence
$\deg\ndt=2$ and $[t^2]\ndt$ is determined, because one knows $\eus N_{\Delta(1)}(1)=\#\anod$.
\par \textbullet \ \ For item~7, there is a unique rank level of maximal size and its size is $3$.
Therefore, $\deg\ndt=3$ and $[t^3]\ndt=1$, which again allows us to compute $[t^2]\ndt$.
\vskip1ex\noindent
A posteriori, it is always true in the abelian case that  
\beq   \label{eq:always-true-ab}
\frac{\eus N'_{\Delta(1)}(1)}{\eus N_{\Delta(1)}(1)}=\frac{\dim\g(1)}{h} .
\eeq
A relationship of this equality to conjectural properties of the reverse operator $\xdo$ is discussed 
in Section~\ref{sect:conj}.

\begin{rmk} The posets $\Delta(1)$ occurring in the abelian case are well known and usually called 
{\it minuscule posets}, see e.g. \cite{pr84,r-s13,stembr}. One can find several Hasse diagrams in 
\cite[Appendix]{stembr}. Since these posets are gaussian, there is a closed formula for the generating 
function for $m$-flags of order ideals in $\Delta(1)$ for any $m\in\BN$ \cite[Sect.\,6]{pr84}; in
particular, a formula for the $\eus M$-polynomial is obtained  if $m=1$.
Therefore, Theorem~\ref{thm:M-polinom-ab} is not really new. But our approach that exploits 
$\BZ$-gradings, the Kostant-Macdonald identity, 
and the set of minimal length representatives $W^0$ seems to be new. It quickly
provides a description of the upper (lower) ideals in $\Delta(1)$ and the corresponding antichains 
(Theorem~\ref{thm:ab-case}).
It can also be used for alternate proofs of other known results on $\anod$.
Anyway, our idea is that the minuscule posets should be treated as the simplest case of weight posets associated with $\BZ$-gradings. \end{rmk}
Recall that the abelian grading of $\g$ related to an admissible $\ap_i\in\Pi$ is defined by the minuscule 
weight $\vp_i^\vee$ of the dual root system $\Delta^\vee$. Let $\sfr(\vp_i^\vee)$ be the 
corresponding representation of the dual Lie algebra $\g^\vee$ and 
$\eus P(\vp_i^\vee)$ its weight poset. 

\begin{thm}[cf.~{\cite[Theorem\,11]{pr84}}]    \label{thm:isom-poset-ab}
The posets $(\anod, \le_{up})$ and $(\eus P(\vp_i^\vee), \curle)$ are naturally isomorphic. 
\end{thm}
\begin{proof}
For a minuscule weight $\vp_i^\vee$, it is known that $\eus P(\vp_i^\vee)=W\vp_i^\vee$
as a set~\cite[Ch.\,VIII,\ \S\,7, n$^0$\,3]{bour7-8}. Since $W(0)$ is the stabiliser of $\vp_i^\vee$, we have
$\eus P(\vp_i^\vee)=\{w(\vp_i^\vee)\mid w\in W^0\}$. By Theorem~\ref{thm:ab-case}, this yields the 
bijection $\Psi: \anod \to \eus P(\vp_i^\vee)$, $\Gamma_w\mapsto w(\vp_i^\vee)$.

Let us  prove that $\Psi$ respects the partial orders. Recall that the order
`$\le_{up}$' corresponds to the inclusion of the corresponding upper ideals,
see Section~\ref{sec:general}.

{\sf\bfseries (a)} \ Suppose that $\Gamma_w$ covers $\Gamma_{w'}$ in $\anod$ for some
$w,w'\in W^0$, i.e., $I_w\supset I_{w'}$ and $\# I_w=(\# I_{w'})+1$. Then
$ \# N(w')= \# N(w)+1$ and $N(w')= N(w)\cup\{\gamma\}$, 
see Theorem~\ref{thm:ab-case}(i). Therefore, $w'=s_\ap w$ for some $\ap\in\Pi$ and
hence $N(w')=N(w)\cup\{w^{-1}(\ap)\}$. In particular, 
$\gamma=w^{-1}(\ap)\in\Delta(1)\subset \Delta^+$. Next, 
\[
  w'(\vp_i^\vee)=s_\ap w(\vp_i^\vee)=
  w(\vp_i^\vee)-(w(\vp_i^\vee),\ap)\ap^\vee=
   w(\vp_i^\vee)-\ap^\vee ,
\]
since 
$(\vp_i^\vee, w^{-1}(\ap))=1$. Whence  $w(\vp_i^\vee)-w'(\vp_i^\vee)=\ap^\vee\in\Pi^\vee$,
i.e., $w(\vp_i^\vee)=\Psi(\Gamma_w)$ covers  $w'(\vp_i^\vee)=\Psi(\Gamma_{w'})$ in 
$\eus P(\vp_i^\vee)$.

{\sf\bfseries (b)} \ Conversely, suppose that $w(\vp_i^\vee)$ covers  $w'(\vp_i^\vee)$ in 
$\eus P(\vp_i^\vee)$ for some $w,w'\in W^0$, i.e., 
$w(\vp_i^\vee)-w'(\vp_i^\vee)=\ap^\vee\in\Pi^\vee$.
Then $(w(\vp_i^\vee)-w'(\vp_i^\vee), \ap)=2$, and the only possibility in the abelian setting 
is that $(w(\vp_i^\vee),\ap)=1$ and
$(w'(\vp_i^\vee), \ap)=-1$. It follows that  $(\vp_i^\vee, w^{-1}(\ap))=1$ and hence 
$w^{-1}(\ap)\in \Delta(1)\subset \Delta^+$. Therefore, $\ell(s_\ap w)=\ell(w)+1$ and 
$N(s_\ap w)=N(w)\cup \{w^{-1}(\ap)\}\subset \Delta(1)$. Consequently, $s_\ap w\in W^0$.
Furthermore,
\[
     s_\ap w(\vp_i^\vee)=w(\vp_i^\vee)-\ap^\vee=w'(\vp_i^\vee) .
\]
Whence $s_\ap w=w'$. This implies that
$N(w')=N(w)\cup \{w^{-1}(\ap)\}$. Thus, 
$I_w=I_{w'}\cup \{w^{-1}(\ap)\}$, and we are done.
\end{proof}

\begin{rmk}
Let `$\leq$' denote the Bruhat order in $W$. It can be shown directly that the posets 
$(\anod, \le_{lo})$ and $(W^0, \leq)$ are isomorphic. This yields an anti-isomorphism of $(W^0, \leq)$ 
and $(\eus P(\vp_i^\vee),\curle)$, which is a particular case of a general result of Proctor, 
see \cite[Prop.\,3]{pr82b}. (Note that Proctor's ``Bruhat order'' is opposite to the usual one!)
\end{rmk}

\section{Antichains and upper ideals in the extra-special case}
\label{sect:ext-case}

\noindent
In this section, we consider the extra-special $\BZ$-grading of $\g$. Now
$\Delta^+=\Delta(0)^+\cup\Delta(1)\cup \Delta(2)$, $\Delta(2)=\{\theta\}$,  and
$\Delta(1)=\{\gamma\in\Delta\mid (\gamma,\theta^\vee)=1\}$ has the following obvious properties:

{\bf --} \ If $\gamma_1,\gamma_2\in\Delta(1)$ and $\gamma_1+\gamma_2$ is a root, then 
$\gamma_1+\gamma_2=\theta$;

{\bf --} \ The reflection $s_\theta\in W$ takes $\gamma\in \Delta(1)$ to $\gamma-\theta\in\Delta(-1)$.
Hence $-s_\theta(\gamma)\in\Delta(1)$ and $\gamma+(-s_\theta(\gamma))=\theta$.
\\
Let $h^*$ be the {\it dual Coxeter number\/} of $\Delta$. By definition, $h^*-1=\hot(\theta^\vee)$, the 
height of $\theta^\vee$ in $\Delta^\vee$. This implies that $h^*\le h$, and $h=h^*$ 
if and only if all the roots have the same length.
By~\cite{rudi}, the total number of roots in $\Delta$ that are not orthogonal to $\theta$ is $4h^*-6$.
It follows that $\#\Delta(1)=\frac{1}{2}(4h^*-6-2)=2h^*-4$.   

\begin{df}  \label{def:lagrangian}
 A subset $S\subset \Delta(1)$ is said to be {\it Lagrangian}, if $\#S=\frac{1}{2}\#\Delta(1)=h^*-2$ and
 there are no roots $\gamma_1,\gamma_2\in S$ such that $\gamma_1+\gamma_2=\theta$.
\end{df}
\noindent
This terminology is justified by the fact that the $\te$-stable subspace of $\g(1)$ corresponding to such $S$ is Lagrangian w.r.t.  the symplectic structure of $\g(1)$ as  $\widetilde{\g(0)}$-module.
\par
 As $ \Delta(1)$ is the disjoint union of pairs 
$\{\gamma,\theta-\gamma\}=\{\gamma, -s_\theta(\gamma)\}$, $S$ is Lagrangian if and only if 
$\Delta(1)\setminus S$ is. Equivalently, 
\beq  \label{eq:lagrange}
  \text{ $S$ is Lagrangian if and only if } \ S =\Delta(1)\setminus -s_\theta(S) .
\eeq

\begin{lm}  \label{lm:prelim-extra}
Let $I\subset\Delta(1)$ be an upper ideal. 
\begin{itemize}
\item[\sf (a)] \ If\/  $\gamma_1+\gamma_2=\theta$ for some\/ $\gamma_1,\gamma_2\in I$, then
there are no pairs $\mu_1,\mu_2\in \Delta(1)\setminus I$ such that $\mu_1+\mu_2=\theta$. In this case,
$\# I>\frac{1}{2}\#\Delta(1)$.
\par
\item[\sf (b)]  If\/ $\mu_1+\mu_2=\theta$ for some\/ $\mu_1,\mu_2\in \Delta(1)\setminus I$, then there are no pairs
$\gamma_1,\gamma_2\in I$ such that $\gamma_1+\gamma_2=\theta$. In this case,
$\# I<\frac{1}{2}\#\Delta(1)$.
\end{itemize}
\end{lm}
\begin{proof}
Assume that  $\gamma_1+\gamma_2=\theta$ and  $\mu_1+\mu_2=\theta$. Since 
$(\gamma_1+\gamma_2, \mu_1+\mu_2)> 0$, we may conclude that, say, $(\gamma_1,\mu_1)>0$. 
Hence $\gamma_1-\mu_1=\mu_2-\gamma_2\in \Delta(0)$. Now, the assumption that 
$\mu_2-\gamma_2\in\Delta(0)^+$ implies that $\mu_2\in I$; while the assumption that 
$\mu_1-\gamma_1\in\Delta(0)^+$ implies that $\mu_1\in I$. These contradictions prove both (a) and (b).
\end{proof}

In what follows, we have to keep track of the length of roots in $\Delta$. Write $\Delta_l$ (resp. $\Pi_l$) 
for the set of {\it all} (resp. {\it simple}) long roots. In the {\bf ADE}-case, all roots are assumed to be both
long and short.
Recall that the highest root $\theta$ is always long and $\theta^\vee$ is always short.

\begin{thm}   \label{thm:extra-case}
{\sf\bfseries (i)}  There is a surjective map $\tau: W^0\to \anod$, $w\mapsto \Gamma_w$, such that 
\[
  \#\tau^{-1}(\Gamma_w)=\begin{cases} 1, & \text{if the upper ideal \ $I(\Gamma_w)$ \ is not Lagrangian}, \\
  2, & \text{if the upper ideal \ $I(\Gamma_w)$ \ is Lagrangian}.  \end{cases}
\]
More precisely, $\Gamma_w=\min(\Delta(1)\setminus N(w))$ and $I_w:=I(\Gamma_w)=\Delta(1)\setminus N(w)$.
\par
{\sf \bfseries(ii)} 
The  upper ideal $I_w$ is Lagrangian {\sl if and only if} $w(\theta)\in\pm \Pi_l$ {\sl if and only if}
$\tau^{-1}(\Gamma_w)=\{w,ws_\theta\}$. In particular, 
$\#\anod=\#W^0-\#\Pi_l=(h-1){\cdot}\#\Pi_l$.
\end{thm}
\begin{proof}
{\bf (i)} \ If $w\in W^0$, then $N(w)\subset \Delta(1)\cup\{\theta\}$. Set $I_w:=\Delta(1)\setminus N(w)$
and $\Gamma_w:=\min(I_w)$. As in the proof of Theorem~\ref{thm:ab-case}, one readily verifies that 
$I_w$ is an upper ideal of  $\Delta(1)$. Hence $\Gamma_w\in \anod$.  This yields the mapping 
$\tau: W^0\to \anod$ with $\tau(w)=\Gamma_w$. Let us prove that $\tau$ is onto and the fibres 
of $\tau$ match the above description.

Let $I\in \eus J_+(\Delta(1))$. To obtain $w\in W^0$ with $I=I_w$, we need two complementary closed subsets of $\Delta^+$ such that one of them contains $\Delta(0)^+\cup I$, and the other contains $\Delta(1)\setminus I$. That is, the only problem is how to handle $\theta$.
By Lemma~\ref{lm:prelim-extra}, there are three possibilities:

{\sf (a)} \ {\it $\theta$ is a sum of two elements from $I$, but not from $\Delta(1)\setminus I$}.
\\
The only suitable pair of complementary closed subsets of $\Delta^+$ is
$\Delta(0)^+\cup I\cup\{\theta\}$ and $\Delta(1)\setminus I$. Here $N(w)=\Delta(1)\setminus I$,
$\# I>\frac{1}{2}\#\Delta(1)$, and $w$ is the unique element of $W^0$ such that $I=I_w$.

{\sf (b)} \ {\it $\theta$ is a sum of two elements from $\Delta(1)\setminus I$, but not from $I$}.
\\
The only suitable pair of complementary closed subsets of $\Delta^+$ is $\Delta(0)^+\cup I$ 
and $(\Delta(1)\setminus I)\cup\{\theta\}$. Here $N(w)=(\Delta(1)\setminus I)\cup\{\theta\}$,
$\# I<\frac{1}{2}\#\Delta(1)$, and $w$ is the unique element of $W^0$ such that $I=I_w$.

{\sf (c)} \  {\it $\theta$ cannot be written as a sum of two elements from either $I$ or $\Delta(1)\setminus I$.}
\\
Then $I$ is Lagrangian and both $\Delta(1)\setminus I$ and $(\Delta(1)\setminus I)\cup\{\theta\}$ 
are the inversion sets of certain elements of $W^0$. That is, $\tau^{-1}(\min(I))$ consists of two 
elements. 
\par
This proves part (i).

{\bf (ii)} \ Let $I\in \eus J_+(\Delta(1))$ be Lagrangian. Then $\tau^{-1}(\min(I))=\{w',w''\}$, where
$N(w')=\Delta(1)\setminus I$ and $N(w'')=(\Delta(1)\setminus I)\cup\{\theta\}$. We claim that
$w''=w's_\theta$. Indeed, $s_\theta$ acts trivially on $\Delta(0)$ and therefore
\[
  N(w's_\theta)=\bigl(\Delta(1)\setminus(-s_\theta N(w')\bigr)\cup\{\theta\}=N(w')\cup\{\theta\}=N(w'') ,
\]
where the second equality holds because $N(w')$ is a Lagrangian subset, cf. Eq.~\eqref{eq:lagrange}. In particular, $w'(\theta)=-w''(\theta)$, and our choice is that $w'(\theta)\in \Delta^+$.
If $w'(\theta)=\delta_1+\delta_2$ is a sum of positive roots, then $\theta=w'^{-1}(\delta_1)+w'^{-1}(\delta_2)$, and the only possibility is that both
roots $w'^{-1}(\delta_1), w'^{-1}(\delta_2)$ belong to $\Delta(1)$. It follows that
$w'^{-1}(\delta_1), w'^{-1}(\delta_2)\in I=\Delta(1)\setminus N(w')$, which contradicts the fact that 
$I$ is Lagrangian. Thus, $w'(\theta)$ must be a (long)  simple root.
\par
The above argument also shows that if $w\in W^0$ and $w(\theta)\not\in \pm \Pi$, then $I_w$ cannot
be Lagrangian. (More precisely, if $w(\theta)\in \Delta^+\setminus \Pi$, then one finds 
$\gamma_1,\gamma_2\in I_w$ such that $\gamma_1+\gamma_2=\theta$; while if 
$w(\theta)\in -(\Delta^+\setminus \Pi)$, then one finds 
$\mu_1,\mu_2\in \Delta(1)\setminus I_w$ such that $\mu_1+\mu_2=\theta$.)
\par
Since $W(0)$ is the stabiliser of $\theta$ in $W$, the mapping $W^0\to \Delta_l$, $w\mapsto w(\theta)$, is one-to-one. Consequently, there are exactly $\#\Pi_l$ Lagrangian ideals.
\par Finally, it follows from~\cite[Ch.\,VI,\ \S\,1.11, Prop.\,33]{bour} that the total number of long roots is 
$h{\cdot}\#\Pi_l$.
\end{proof}

It is sometimes convenient to think of $\tau$ as the map from $W^0$ to $\eus J_+(\Delta(1))$ with
$\tau(w)=\Delta(1)\setminus N(w)$. As in the abelian case,
we can  give a characterisation of antichains associated with upper ideals
in $\Delta(1)$ via the corresponding elements of $W^0$.

\begin{thm}  \label{thm:next-extra}
For $I\in \eus J_+(\Delta(1))$, let $w_I$ be a corresponding element of\/ $W^0$. (For the Lagrangian 
ideals $I$, we specify the choice of $w_I$ below.) Then
\par
{\sf\bfseries (i)} \ $\gamma\in \min(I)$ if and only if $w_I(\gamma)\in\Pi$ \ (for a Lagrangian $I$, one has 
to choose $w_I$ such that $w_I(\theta)\in -\Delta^+$).
\par
{\sf\bfseries (ii)} \ $\gamma\in \max(\Delta(1)\setminus I)$ if and only if $w_I(\gamma)\in -\Pi$ \ 
(for a Lagrangian $I$, one has to choose $w_I$ such that  $w_I(\theta)\in \Delta^+$).
\end{thm}
\begin{proof}
{\bf (i)} If $\gamma\not\in\min(I)$, then $\gamma=\gamma'+\delta$ for some 
$\gamma'\in I$ and  $\delta\in\Delta(0)^+$.
Whence $w_I(\gamma)=w_I(\gamma')+w_I(\delta)$ is a sum of positive roots, i.e.,
$w_I(\gamma)\not\in\Pi$ (for any choice of $w_I$ if
$I$ is Lagrangian!). 

Conversely, assume that $w_I(\gamma)\not\in\Pi$, i.e., $w_I(\gamma)=\delta_1+\delta_2$ is a sum of positive roots. Then
\[
    \gamma=w_I^{-1}(\delta_1)+w_I^{-1}(\delta_2) \in \Delta(1) .
\]
Set $\mu_i=w_I^{-1}(\delta_i)$. There are two possibilities for $\mu_1,\mu_2$:

{\sl \bfseries (a)} \ $\mu_1\in\Delta(0)$ and  $\mu_2\in\Delta(1)$.  \\
Then 
$\mu_2\in\Delta(1)\setminus N(w_I)=I$ and $\mu_1\in\Delta(0)^+$, since $w\in W^0$.
Hence $\gamma=\mu_1+\mu_2\not\in\min(I)$.

{\sl \bfseries (b)} \  $\mu_1\in\Delta(-1)$ and  $\mu_2=\theta\in\Delta(2)$. \\
Then $w_I(\theta)=\delta_2$ is positive. Hence $\theta\not\in N(w_I)$ and therefore 
$\# N(w_I)\le \frac{1}{2}\#\Delta(1)$, i.e., $\# I \ge\frac{1}{2}\#\Delta(1)$. 
\\
{\textbullet } \ For $I$ Lagrangian, we agree to choose $w_I$ such that $w_I(\theta)< 0$, which 
eliminates such a possibility for $\mu_1$ and $\mu_2$.
\\
{\textbullet } \ Assume that $\# I >\frac{1}{2}\#\Delta(1)$. We have $\gamma+(-\mu_1)=\theta$ and 
$-\mu_1\in \Delta(1)\setminus I$. If $\gamma\in\min(I)$, then $I':=I\setminus \{\gamma\}$ is again an  
upper ideal. Here $\gamma,-\mu_1 \not\in I'$ and their sum is $\theta$. By 
Lemma~\ref{lm:prelim-extra}(ii), we then have $(\#I)-1=\#I' < \frac{1}{2}\#\Delta(1)$. And this contradicts the fact that $\#\Delta(1)$ is even. 
\par
Thus, $\gamma\not\in\min(I)$ in all cases. 
\\[.7ex]
{\bf (ii)}  The proof here is similar and ``dual'' to the preceding part. For instance, 
at some point one refers to the ``dual'' 
fact that if $\gamma\in\max(\Delta(1)\setminus I)$, then $I\cup\{\gamma\}$ is again an upper ideal.
The details are left to the reader.
\end{proof}

Having computed the number of antichains (upper ideals) in $\Delta(1)$, we turn to computing the
$t$-analogues $\mdt$ and $\ndt$. Although the relationship between the upper ideals of $\Delta(1)$
and $W^0$ appeared to be more involved in the extra-special case than in the abelian one, 
the formula for $\mdt$ remains just the same!

\begin{thm}   \label{thm:mdt-extra}
For the extra-special gradings, we have \ 
$     \displaystyle
   \eus M_{\Delta(1)}(t)=\prod_{\gamma\in\Delta(1)}\frac{1-t^{\hot(\gamma)+1}}{1-t^{\hot(\gamma)}} .
$
\end{thm}
\begin{proof}  Let $P(t)$ denote the right hand side of the formula. {\sl A priori}, $P(t)$ is only a rational 
function, and our first goal is to prove that $P(t)$ is a polynomial. 
Using the Kostant-Macdonald identity (cf. Theorem~\ref{thm:M-polinom-ab}), the
decomposition $\Delta^+=\Delta(0)^+\cup\Delta(1)\cup \{\theta\}$, and the equality $\hot(\theta)=h-1$,
we obtain
\beq  \label{eq:in-thm-1}
  P(t)=\frac{\bow(t)}{\boldsymbol{W\!(0)}(t)}\cdot \frac{1-t^{h-1}}{1-t^h}=
  {\boldsymbol{W^0}(t)}\cdot \frac{1-t^{h-1}}{1-t^h}={\boldsymbol{W^0}(t)}-t^{h-1}{\boldsymbol{W^0}(t)}
  \cdot \frac{1-t}{1-t^h} .
\eeq
It follows that $P(t)\in \BZ[t]$ if and only if 
${\boldsymbol{W^0}(t)}{\cdot} \displaystyle\frac{1-t}{1-t^h}\in \BZ[t]$. Recently, I proved that the latter 
is related to the {\it Lusztig $t$-analogue\/} of the zero weight multiplicity in $\sfr(\theta^\vee)$,
the representation of the dual Lie algebra $\g^\vee$ with highest weight $\theta^\vee$. Namely, let 
$\mathfrak M^0_{\theta^\vee}(t)$ be the above-mentioned $t$-analogue. It is a polynomial with 
nonnegative coefficients and $\mathfrak M^0_{\theta^\vee}(1)=m^0_{\theta^\vee}$, the respective 
weight multiplicity in $\sfr(\theta^\vee)$. 
By \cite[Cor.\,3.6]{all-q-an}, we have
\beq  \label{eq:in-thm-2}
    {\boldsymbol{W^0}(t)}{\cdot} \displaystyle\frac{1-t}{1-t^h}=
    \frac{\mathfrak M^0_{\theta^\vee}(t)} {t^{h-\hot(\theta^\vee)}} ,
\eeq
where $\hot(\theta^\vee)$ is the height of $\theta^\vee$ in the dual root system $\Delta^\vee$.
(Note that the indeterminate in \cite{all-q-an} is denoted by $q$ in place of $t$ and 
our $\boldsymbol{W^0}(t)$ is  $t_0(q)/t_\theta(q)$ therein.)
As the left hand side has no pole at $t=0$, both parts are polynomials in $t$, which proves that 
$P(t)$ is a polynomial, too.  Once we know that $P(t)$ is a polynomial, it follows from the very definition 
of it that $\deg P(t)=\#\Delta(1)=2h^*-4$ and $P(t)$ is palindromic.

It remains to prove that $[t^i]P(t)$ is the number of upper ideals of cardinality $i$.
Set $Q(t)=\displaystyle \frac{\mathfrak M^0_{\theta^\vee}(t)} {t^{h-\hot(\theta^\vee)}}$.
(Although we do not need it directly,  we note that $\deg \mathfrak M^0_{\theta^\vee}(t)=\hot(\theta^\vee)=
h^*-1$, hence $\deg Q(t)=2h^*-h-2$.)
By Eq.~\eqref{eq:in-thm-1} and \eqref{eq:in-thm-2}, we have
\beq    \label{eq:in-thm-3}
   P(t)={\boldsymbol{W^0}(t)}-t^{h-1}Q(t) .
\eeq
As $h-1\ge h^*-1$, it follows from Eq.~\eqref{eq:in-thm-3} that 
\[
     [t^i]P(t)=[t^i]{\boldsymbol{W^0}(t)}=\#\{w\in W^0\mid \ell(w)=i\} \quad \text{for $i\le h^*-2$}.
\]
If $w\in W^0$ and $\ell(w)\le h^*-2=\frac{1}{2}\#\Delta(1)$, then Theorem~\ref{thm:extra-case} implies
that
$\theta\not\in N(w)$ and $N(w)\subset \Delta(1)$. Hence for $i\le h^*-2$, $[t^i]P(t)$ equals the number of
upper ideals of cardinality $\#\Delta(1)-i$. Since $\mdt$ is palindromic and of degree $\#\Delta(1)$
(Lemma~\ref{lem:M-polinom-P(V)}), the latter is also
the number of upper ideals of cardinality $i$.  Thus, $\mdt$ and $P(t)$ are palindromic of equal degrees
$2h^*-4$, and $[t^i]P(t)=[t^i]\mdt$ for $i\le h^*-2$. Consequently, $P(t)\equiv \mdt$.
\end{proof}

\begin{cl}   \label{cor:card-extra}
In the extra-special case, we have 
$\#\anod=\eus M_{\Delta(1)}(1)=\prod_{\gamma\in\Delta(1)}\frac{\hot(\gamma)+1}{\hot(\gamma)}$.
\end{cl}

\begin{rmk}  \label{rem:yet-another}
Yet another formula for $\mdt$, which follows from Eq.~\eqref{eq:in-thm-1} and \eqref{eq:in-thm-2}, is
\[
    \mdt=\frac{\mathfrak M^0_{\theta^\vee}(t)} {t^{h-\hot(\theta^\vee)}}\cdot
    \frac{1-t^{h-1}}{1-t} .
\]
The weight multiplicity $m^0_{\theta^\vee}=\mathfrak M^0_{\theta^\vee}(1)$ equals the number of short simple roots of $\g^\vee$, i.e.,
the number of long simple roots of $\g$, i.e., $\#\Pi_l$. Hence we again obtain the equality
$\#\anod=\eus M_{\Delta(1)}(1)=\#\Pi_l{\cdot}(h-1)$.
\par
If $\Delta\in \{\mathbf{ADE}\}$, then $\g^\vee\simeq\g$, \ $\hot(\theta^\vee)=h-1$, and
$\mathfrak M^0_{\theta^\vee}(t)=\sum_{i=1}^n t^{m_i}$, where $m_1,\dots,m_n$ are the exponents of
$W$. Here we obtain  a very simple explicit formula
\[
  \mdt=(\sum_{i=1}^n t^{m_i-1})(1+t+\ldots +t^{h-2}) .
\]
\end{rmk}

\begin{thm}   \label{thm:ndt-extra}
In the extra-special case,  $\deg\ndt\le 3$, i.e., if\/ $\Gamma\in\anod$, then
$\#\Gamma\le 3$. 
\end{thm}
\begin{proof}
Recall that $\g(1)$ is a simple $\g(0)$-module unless $\Delta$ is of type $\GR{A}{n}$, and therefore
in all these cases $\Delta(1)$ is rank symmetric, rank unimodal, and Sperner (see 
Lemma~\ref{lem:sperner-sl2}). 
Therefore, $\deg\ndt=\#\Delta(1)_i$, where $\Delta(1)_i$ is
the set of roots of height $i$ in $\Delta(1)$ and 
$i$ is a middle rank of $\Delta(1)$. As $\hot(\theta)=h-1$ and the unique 
element of $\Delta^+$ covered by $\theta$ belongs to 
$\Delta(1)$, the roots in $\Delta(1)$ have the height between $1$ and $h-2$. Furthermore,
$h$ is even, if $\g$ is not of type $\GR{A}{n}$. Hence the two middle ranks are $(h-2)/2$ and $h/2$,
and $\deg\ndt=\#\Delta(1)_{h/2}=\#\Delta(1)_{(h-2)/2}$.
Assume that $\Delta(1)_{h/2}=\{\gamma_1,\dots,\gamma_k\}$. Since $\gamma_i\pm\gamma_j$ is not a 
root, all roots in $\Delta(1)_{h/2}$ are pairwise orthogonal. Consequently, $\mu:=\theta-\sum_{i=1}^k\gamma_i$ is a root.
Because $\theta\in\Delta(2)$ and $\gamma_i\in\Delta(1)$, we have $\mu\in\Delta(2-k)$. Therefore, 
$k\le 4$, and if $k=4$, then $\mu=-\theta$. However, comparing heights in the equality
$2\theta=\sum_{i=1}^4\gamma_i$, we see that this is impossible! (For, $2\hot(\theta)=2h-2$, whereas
$\sum_{i=1}^4\hot(\gamma_i)=2h$.) Thus, $k\le 3$.

For $\GR{A}{n}$, the poset $\Delta(1)$ is the disjoint union of two chains, 
$\Delta(1)=\eus C_{n-1}\sqcup \eus C_{n-1}$,
whence 
$\deg\ndt = 2$. 
\end{proof}

\noindent
Thus,  $\ndt=1+N_1t+N_2t^2+N_3t^3$, $N_1=\#\Delta(1)$, and $\eus N_{\Delta(1)}(1)=\#\anod$. 
Therefore, to completely determine $\ndt$,  only one more condition is needed.
Below, we compute $N_2$ in the {\bf ADE}-case and thereby provide a nice uniform expression for $\ndt$ . 
To this end,
we begin with a general look at the two-element antichains in $\Delta(1)$ (= $2$-antichains).

Let $\{\gamma_1,\gamma_2\}$ be a $2$-antichain in $\Delta(1)$. Then $(\gamma_1,\gamma_2)\le 0$ and there are two possibilities:
\vskip.7ex
${(a_1)}$  $(\gamma_1,\gamma_2)= 0$. \ Such an antichain is said to be {\it orthogonal}.

$(a_2)$  $(\gamma_1,\gamma_2)< 0$. \ Then $\gamma_1+\gamma_2=\theta$ and such an 
antichain is said to be {\it summable}.
\\[.7ex]
If $\{\gamma_1,\gamma_2\}$ is an orthogonal antichain, then $\gamma_1+\gamma_2-\theta$ is a root, necessarily in $\Delta(0)$; while in the second case,  $\gamma_1+\gamma_2-\theta=0$.
Therefore, one obtains a general map 
\[
    \varkappa: \anod_{\langle 2\rangle}\to \Delta(0)\cup\{0\} , 
\]
where $\anod_{\langle 2\rangle}$ is the set of $2$-antichains in $\Delta(1)$ and 
$\varkappa(\{\gamma_1,\gamma_2\})=\gamma_1+\gamma_2-\theta$. 

\begin{thm}  \label{thm:extra-ndt-ADE}
If \ $\Delta\in \{\mathbf{ADE}\}$, then {\rm\sf (i)} for any $\mu\in\Delta(0)$, there is a unique 
(orthogonal) $2$-antichain $\Gamma$ such that $\varkappa(\Gamma)=\mu$, {\rm\sf (ii)} the number 
of summable antichains is equal to $\rk\g-1$. \\ In particular, $\varkappa$ is onto and 
$N_2=\dim\g(0)-1=\dim\widetilde{\g(0)}$.
\end{thm}
\begin{proof}
The  argument below is not entirely case-free.
On the other hand, a complete case-by-case checking is also possible. For this
reason, we only outline some steps and their status ({\sf case-free} or {\sf case-by-case}).

{\bf 1.} \ If $\{\gamma_1,\gamma_2\}$ is an antichain in $\Delta(1)$, 
then so is $\{\theta-\gamma_1,\theta-\gamma_2\}$, and
$\varkappa(\{\gamma_1,\gamma_2\})=-\varkappa(\{\theta-\gamma_1,\theta-\gamma_2\})$.
Therefore $\mu\in\Ima(\varkappa)$ if and only if $-\mu\in\Ima(\varkappa)$, and it suffices to consider 
only $\mu\in\Delta(0)^+\cup\{0\}$.

{\bf 2.}  (Uniqueness) Assume that $\{\gamma_1,\gamma_2\}$ and $\{\gamma'_1,\gamma'_2\}$ lie
in $\varkappa^{-1}(\mu)$, i.e., $\gamma_1+\gamma_2=\gamma'_1+\gamma'_2=\theta+\mu$ for some
$\mu\in\Delta(0)^+$. All the roots involved have the same length and 
$(\gamma_1,\gamma_2)=(\gamma'_1,\gamma'_2)=0$. Since
$(\gamma_1,\gamma_1)+(\gamma_2,\gamma_2)=(\gamma_1+\gamma_2)^2=
   (\gamma_1+\gamma_2, \gamma'_1+\gamma'_2)$, 
we have $(\gamma_i,\gamma'_j)>0$ for all $i,j\in \{1,2\}$. In particular, 
$\gamma_i-\gamma'_j\in\Delta(0)$. Set $\nu=\gamma'_1-\gamma_1$ and 
$\eta=\gamma'_1-\gamma_2=\gamma_1-\gamma'_2$. W.l.o.g. we may assume that 
$\nu\in\Delta(0)^+$, hence $\gamma'_1\curge\gamma_1$ in $\Delta(1)$.
Now, if $\eta\in\Delta(0)^+$, then $\gamma'_1\curge\gamma_1\curge\gamma'_2$, i.e., 
$\{\gamma'_1,\gamma'_2\}$ is not an antichain; if $-\eta\in\Delta(0)^+$, then $\gamma_2\curge\gamma'_1\curge\gamma_1$, i.e., 
$\{\gamma_1,\gamma_2\}$ is not an antichain. These contradictions prove that 
$\#\varkappa^{-1}(\mu)\le 1$.

{\bf 3.} (Existence) \ (i) ({\it Base}) \ if $\mu$ is the highest root of  an irreducible subsystem of $\Delta(0)$, then
$\varkappa^{-1}(\mu)\ne\varnothing$. \  [{\sf case-by-case}]
\par
(ii) \ ({\it Induction step}) if $\varkappa^{-1}(\mu)\ne\varnothing$ and  $\mu-\ap\in \Delta(0)^+\cup\{0\}$ for some $\ap\in\Pi(0)$, 
then $\varkappa^{-1}(\mu-\ap)\ne\varnothing$.  \ [{\sf case-free}]
\\[.7ex]
These three steps prove everything concerning the orthogonal antichains, and also show that 
summable antichains exist; hence $\varkappa$ is onto.

 {\bf 4.} The assertion on the number of summable antichains occurs as a by-product of certain 
results of mine related to abelian $\be$-ideals in $\ut^+$. This will appear elsewhere. Actually, those 
results provide a one-to-one correspondence between the summable antichains and the edges of 
the Dynkin diagram. [{\sf case-free}]
\par  This completes our outline. 
\end{proof}   

As an illustration to the proof, we point out all summable antichains for
$\GR{E}{6}$. The numbering of simple roots is \  
$\text{\begin{E6}{1}{2}{3}{4}{5}{6}\end{E6}}$ \ and $(n_1n_2\dots n_6)$ stands for the root
$\gamma=\sum_{i=1}^6 n_i\ap_i$. In particular,  $\theta=(123212)$ and $\gamma\in\Delta(1)$ if and only if 
$n_6=1$. Then the summable antichains in $\Delta(1)$ are:\\
$\{111001, 012211\}, \{111101, 012111\}, \{111111, 012101\}, \{011111, 112101\}, \{001111, 122101\}$. 
\par
Note also that for $\GR{A}{n}$ and $\GR{D}{n}$, everything in Theorem~\ref{thm:extra-ndt-ADE} can 
explicitly be verified, using the usual $\{\esi_i\}$ presentation of the roots.

\begin{cl}  \label{cl:extra-ndt-ADE}
If \ $\Delta\in \{\mathbf{ADE}\}$, then
\beq    \label{eq:ndt-extra-ADE}
   \ndt=1+\dim\g(1){\cdot}t+(\dim\g(0)-1){\cdot}t^2+(\dim\g(1)-2\rk\g+2){\cdot}t^3 .  
\eeq
\end{cl}
\noindent
For $\GR{A}{n}$, we have $\dim\g(1)=2\rk\g-2=2n-2$ and $\deg\ndt=2$ (if $n>1$). 
It also follows from Eq.~\eqref{eq:ndt-extra-ADE} that \ 
$\displaystyle\frac{\eus N'_{\Delta(1)}(1)}{\eus N_{\Delta(1)}(1)}=
\frac{2\dim\g-6\rk\g}{\dim\g-2\rk\g}=\frac{2h-4}{h-1}$, \ since $\dim\g=(h+1)\rk\g$.

\begin{rem}
Some steps in the proof of Theorem~\ref{thm:extra-ndt-ADE} go through for any $\Delta$ (e.g. {\bf 1.}
and {\bf 2.}), but $\varkappa$ is no longer onto in general.
Although $\Ima(\varkappa)$ can explicitly be described in each non-simply laced case, we are unable to
infer from it a general characterisation of  $\Ima(\varkappa)$. Here is the list 
of $\eus N$-polynomials for the remaining root systems:

$\GR{B}{n}$, $n\ge 2$: \quad  $\ndt=1+2(2n-3)t+(n-2)(2n-3)t^2$;

$\GR{C}{n}$, $n\ge 2$: \quad  $\ndt=1+(2n-2)t$;

$\GR{F}{4}$: \quad  $\ndt=1+14t+7t^2$;

$\GR{G}{2}$: \quad  $\ndt=1+4t$;
\end{rem}

{\sl A posteriori},  for all irreducible root systems $\Delta$, we have
\beq     \label{eq:always-true-extra}
  \displaystyle\frac{\eus N'_{\Delta(1)}(1)}{\eus N_{\Delta(1)}(1)}=\frac{\#\Delta(1)}{h-1}=
  \frac{2h^*-4}{h-1} .
\eeq
As in the abelian case (cf. Theorem~\ref{thm:isom-poset-ab}), we provide a realisation of 
$(\anod, \le_{up})$ via the representation of $\g^\vee$ associated with the defining element of 
the $\BZ$-grading, which is now $\theta^\vee\in\Delta^\vee$. A new phenomenon is that 
$\sfr(\theta^\vee)$ is not \textsf{wmf}. Indeed, $\eus P(\theta^\vee)=\{0\}\cup (\Delta_l)^\vee$, i.e., it
contains zero and the short roots of $\Delta^\vee$. The multiplicity of $0$ equals $\#\Pi_l$, i.e., the 
number of short simple roots in $\Delta^\vee$, and all other weights are of multiplicity one (and form a 
sole $W$-orbit). The number of nonzero weights is $\# W(\theta^\vee)=\#W^0=h{\cdot}\#\Pi_l$. To adjust
it to the relationship between $\anod$ and $W^0$ occurring in Theorem~\ref{thm:extra-case}, 
we do the following:

\textbullet \quad forget about  $0\in \eus P(\theta^\vee)$, i.e., stick to $(\Delta_l)^\vee$;

\textbullet \quad identify $\ap^\vee$ and $-\ap^\vee$ in $(\Delta_l)^\vee$ for $\ap\in\Pi_l$.
\\
The resulting set is denoted by $(\Delta_l)^\vee\!/\!\sim$. It has the natural partial order induced from 
$(\eus P(\theta^\vee),\curle)$. For, we change nothing in the upper (positive) part
$(\Delta_l^+)^\vee$ and in the lower (negative) part $-(\Delta_l^+)^\vee$ of $\eus P(\theta^\vee)$. 
And we only identify "element-wise" the subsets $\min((\Delta_l^+)^\vee)=(\Pi_l)^\vee$ and 
$\max(-(\Delta_l^+)^\vee)=(-\Pi_l)^\vee$.

\begin{thm}    \label{thm:isom-poset-extra}
The posets $(\anod, \le_{up})$ and $((\Delta_l)^\vee\!/\!\sim, \curle)$ are naturally isomorphic. 
\end{thm}
\begin{proof}
In Theorem~\ref{thm:extra-case}, we have defined the surjective map $\tau: W^0\to \anod$. Recall that
$\#\tau^{-1}(\Gamma)\le 2$ and $\#\tau^{-1}(\Gamma)= 2$  if and only if 
$I(\Gamma)\in\eus J_+(\Delta(1))$  is Lagrangian. Moreover, if $\tau^{-1}(\Gamma)=\{w',w''\}$, then 
$w'(\theta)=-w''(\theta)\in \pm\Pi_l$. Therefore, 

{\bf --} \quad if $I(\Gamma)$ is Lagrangian, then $\tau^{-1}(\Gamma)(\theta)=\{\ap,-\ap\}$ for some
$\ap\in\Pi_l$;

{\bf --} \quad if $I(\Gamma)$ is not Lagrangian, then $\tau^{-1}(\Gamma)(\theta)=
\mu\in \Delta_l \setminus (\Pi\cup -\Pi)$.
\\
Thus, the map $\Psi: \anod \to (\Delta_l)^\vee\!/\!\sim$, \ 
$\Gamma \mapsto \tau^{-1}(\Gamma)(\theta)$, is well-defined and onto. Let us 
prove that $\Psi$ respects the partial orders. 

{\sf\bfseries (a)} \ Suppose that $\Gamma_1$ covers $\Gamma_2$ in $(\anod, \le_{up})$, i.e.,
$I(\Gamma_1)=I(\Gamma_2)\cup\{\gamma\}$ for some $\gamma\in\Delta(1)$. Our 
{\it\bfseries goal\/} is to obtain the
relation $N(w_2)=N(w_1)\cup\{\gamma\}$ for some $w_i\in \tau^{-1}(\Gamma_i)$. If this is the case,
then $w_2=s_\ap w_1$ and $\gamma=w_1^{-1}(\ap)$ for some $\ap\in\Pi$. Furthermore,
\[
    w_2(\theta^\vee)=s_\ap w_1(\theta^\vee)=w_1(\theta^\vee)-(w_1(\theta^\vee),\ap)\ap^\vee=
    w_1(\theta^\vee)-\ap^\vee ,
\]
which implies that $\Psi(\Gamma_1)$ covers $\Psi(\Gamma_2)$ in $(\Delta_l)^\vee\!/\!\sim$.
It is important here that both  $w_1(\theta^\vee)$ and $w_2(\theta^\vee)$ lie in
either $(\Delta^+_l)^\vee$ or $(-\Delta^+_l)^\vee$.

{\it\bfseries How to reach that goal}:
\par
${\bf (a_1)}$ If neither $I(\Gamma_1)$ nor $I(\Gamma_2)$ is Lagrangian, then $w_1,w_2$ are uniquely
determined and the required relation holds automatically. In particular, if  
$\#I(\Gamma_1)<\frac{1}{2}\#\Delta(1)$, then both $N(w_1)$ and $N(w_2)$ contain $\theta$ and 
$w_1(\theta^\vee), w_2(\theta^\vee)\in (\Delta^+_l)^\vee$; 
while if $\#I(\Gamma_2)>\frac{1}{2}\#\Delta(1)$, then both $N(w_1)$ and $N(w_2)$ do not contain 
$\theta$ and $w_1(\theta^\vee), w_2(\theta^\vee)\in (-\Delta^+_l)^\vee$.
\par
${\bf (a_2)}$ If one of the ideals is Lagrangian (note that there are two different possibilities for this), 
then the "non-Lagrangian" element $w_i$ is uniquely determined, and for the Lagrangian ideal 
$I(\Gamma_{\ov{i}})$ we choose the element $w_{\ov{i}}\in\tau^{-1}(\Gamma_{\ov{i}})$ such that 
$w_i(\theta)$ and $w_{\ov{i}}(\theta)$ have the same sign. This choice guarantee us that 
$N(w_i)$ and $N(w_{\ov{i}})$ simultaneously contain or do not contain $\theta$. 

{\sf\bfseries (b)} \ Conversely, if $w(\vp_i^\vee)$ covers  $w'(\vp_i^\vee)$ in 
$(\Delta_l)^\vee\!/\!\!\sim$, then, as in part {\sf\bfseries (a)}, the argument goes through along the lines 
presented in Theorem~\ref{thm:isom-poset-ab}, with amendments caused by the presence of 
$\{\theta\}=\Delta(2)$ and  relation `$\sim$' in $(\Delta_l)^\vee$.
We omit the details.
\end{proof}

\section{Conjectures and examples}
\label{sect:conj}

So far, almost nothing is said about the reverse operators for posets $\Delta(1)$. This will be fixed 
below. Numerous calculations performed in the abelian, extra-special, and some other cases suggest 
that the reverse operators $\xdo$ have very good properties similar to those of $\fX_{\Delta^+}$ 
(see Introduction), and also a new one. But outside the realm of the \textsf{wmf} representations associated 
with $\BZ$-gradings some of these properties certainly fail.

\begin{conj}   \label{conj:general1}
For any $\BZ$-grading of $\g$, we have $\displaystyle
   \eus M_{\Delta(1)}(t)=\prod_{\gamma\in\Delta(1)}\frac{1-t^{\hot(\gamma)+1}}{1-t^{\hot(\gamma)}}$.
In particular, $\#\anod=\eus M_{\Delta(1)}(1)=\prod_{\gamma\in\Delta(1)}\frac{\hot(\gamma)+1}{\hot(\gamma)}$.
\end{conj}

This conjecture readily reduces to $1$-standard gradings. {\sl First}, the $\g(0)$-module $\g(1)$ is  determined by $\Pi(1)$ and does not depend on $\Pi({\ge}2)$. Having removed from the Dynkin diagram the nodes (simple roots) in $\Pi({\ge}2)$, we get a standard $\BZ$-grading of a semisimple subalgebra
$\es\subset\g$, with the same poset $\Delta(1)$. {\it Second}, each simple $\es(0)$-submodule of
$\es(1)$ is associated with a $1$-standard $\BZ$-grading of a certain simple factor of $\es$, and we can use the relevant assertion of Lemma~\ref{lem:disjoint-union}.

In view of Theorems~\ref{thm:M-polinom-ab} and \ref{thm:mdt-extra}, Conjecture~\ref{conj:general1} 
holds for the abelian and extra-special gradings. Furthermore, it is true if $\Delta(1)$ is the direct product of at most three chains, see Remark~\ref{rmk:plane-partit}. This covers all but one $1$-standard
$\BZ$-gradings for series $\GR{A}{n}$, $\GR{B}{n}$, and $\GR{C}{n}$, see Example~\ref{ex:flaas}.
Conjecture~\ref{conj:general1} has also a natural counterpart for arbitrary irreducible \textsf{wmf} representations. 
If $\eus P(V)$ is the weight poset of an irreducible \textsf{wmf} representation $V$ and 
$r: \eus P(V) \to \BN$ is the tuned rank function, then one might suggest that 
\beq    \label{eq:arb-wmf-suggest}
   \eus M_{\eus P(V)}(t)=\prod_{\nu\in\eus P(V)}\frac{1-t^{r(\nu)+1}}{1-t^{r(\nu)}} \ \text{ and hence } \ 
   \#\eus P(V)=\prod_{\nu\in\eus P(V)}\frac{r(\nu)+1}{r(\nu)} .
\eeq
However, it can happen for some $V$ that the first product is {\bf not} a polynomial and the second 
product in {\bf not} an integer, see Example~\ref{ex:boolean-cubes}(1) below.

An interesting feature of the polynomials $\mdt$ is that they seem to provide a nice 
illustration to the ``$t=-1$ phenomenon'' of Stembridge~\cite{stembr}, which is a particular case of the 
cyclic sieving phenomenon~\cite{sagan}. Recall that, for any $I\in \eus J_+(\Delta(1))$, we 
have defined the dual ideal $I^*=\Delta(1)\setminus w_0(I)$, where $w_0\in W(0)$ is the longest 
element. 

\begin{conj}   \label{conj:general1.5}
The value $\eus M_{\Delta(1)}(-1)$ equals the number of upper ideals $I$ such that $I^*=I$.
\end{conj}

\noindent
Results of \cite{stembr} confirm this assertion in the abelian case. Using the formula in Remark~\ref{rem:yet-another},
we can also prove it in the 
extra-special case. But, the main challenge is to provide conceptual proofs for all (at least, some of) 
the previous and subsequent conjectures!

\begin{conj}   \label{conj:general2}
Let $\g=\bigoplus_{i\in\BZ}\g(i)$ be a $1$-standard $\BZ$-grading (hence $\g(1)$ is a simple $\g(0)$-module and $\widetilde{\g(0)}$ is semisimple). 
\begin{itemize}
\item[\sf\bfseries (i)] \ If\/ $d_1=\max\{\hot(\gamma)\mid \gamma\in \Delta(1)\}$, then $\ord(\xdo)=d_1+1$.
\item[\sf\bfseries (ii)] \ the average value of the size of antichains in any $\xdo$-orbit is the same and 
equals \\ $\displaystyle\frac{\dim\g(1)}{d_1+1}=\frac{\#\Delta(1)}{\ord(\xdo)}$;
\item[\sf\bfseries (iii)] \ the average value of the size of upper ideals in any $\xdo$-orbit is the same and 
equals $\dim\g(1)/2$.
\end{itemize}
\end{conj}

Part (iii) above is a new property that has no counterpart for $\eus P=\Delta^+$.
\\
In the abelian (resp. extra-special) case, we have $d_1=h-1$ (resp. $d_1=h-2$). Therefore, 
Conjecture~\ref{conj:general2}(i) claims that in these cases $\ord(\xdo)$ is equal to $h$ and $h-1$, 
respectively. 

\begin{ex}   \label{ex:stand-orbit} (cf. \cite[Lemma\,1.1]{oper-X})
Let $\eus P=\bigsqcup_{j=1}^d \eus P_i$ be a graded poset such that
$\max(\eus P)=\eus P_d$ and $\min(\eus P)=\eus P_1$, then $\fX_\eus P$ always has an orbit of size
$d+1$. Namely, we have $\fX_\eus P(\eus P_i)=\eus P_{i-1}$ (with $\eus P_0=\varnothing$) and 
$\{\varnothing, \eus P_d,\dots,\eus P_1\}$ is an $\fX_\eus P$-orbit. 
Moreover, the average value of the size of antichains in this orbit equals $\#\eus P/(d+1)$. Even more,
if $\eus P$ is rank symmetric, then the average value of the size of upper ideals in this orbit
equals $\#\eus P/2$. 
\par
Since all these properties hold for $\eus P=\Delta(1)$, we get a motivating example for the whole Conjecture~\ref{conj:general2}.
\end{ex}
\begin{ex}  \label{ex:e6-e7} 
Straightforward computations for some abelian gradings show that
\par
\textbullet\quad For $\g=\GR{E}{6}$ and $\Pi(1)=\{\ap_1\}$ (item~6 in Section~\ref{sect:ab-case}),  
$\xdo$ has three orbits of sizes 12,\,12,\,3. We also know here that $\#\anod=27$. Hence $\ord(\xdo)=12$.
\par
\textbullet\quad 
For $\g=\GR{E}{7}$ and $\Pi(1)=\{\ap_1\}$ (item~7 in Section~\ref{sect:ab-case}),  $\xdo$ has four orbits of sizes 18,\,18,\,18,\,2. We also know here that $\#\anod=56$. Hence $\ord(\xdo)=18$. 
\par
\textbullet\quad For $\g=\GR{D}{n}$ and $\Pi(1)=\{\ap_1\}$ (item~5 in Section~\ref{sect:ab-case}), we
denote $\Delta(1)$ by $\eus D_{n-1}$. The corresponding Hasse diagram is

\begin{center}   $\eus H(\eus D_{n-1})$: \qquad
\begin{picture}(230,25)(0,3)
\multiput(10,10)(20,0){2}{\color{MIXT}\circle*{4}}
\multiput(70,10)(20,0){2}{\color{MIXT}\circle*{4}}
\multiput(130,10)(20,0){2}{\color{MIXT}\circle*{4}}
\multiput(190,10)(20,0){2}{\color{MIXT}\circle*{4}}
\multiput(110,0)(0,20){2}{\color{MIXT}\circle*{4}}

\multiput(11,10)(60,0){4}{\vector(1,0){18}}

\multiput(91,11)(20.5,-10.5){2}{\vector(2,1){18}}
\multiput(91,9)(20.5,10.5){2}{\vector(2,-1){18}}

\put(45,7){$\cdots$}
\put(165,7){$\cdots$}
\put(-5,5){{\small $\ap_1$}}
\put(215,5){{\small $\theta$}}
\end{picture}
\end{center}
\vskip.5ex\noindent
(The number of nodes is $2n-2$.) Here $\#\anod=2n$ and $\xdo$  has two orbits of sizes $2$ and 
$2n-2$, i.e., $\ord(\xdo)=2n-2$.
\\
The other assertions of Conjecture~\ref{conj:general2} are also satisfied in these three cases.
\end{ex}

\begin{ex}  \label{ex:flaas}
Recall that $\eus C_k$ is a $k$-element chain. For  $\eus P=\eus C_k\times \eus C_m$, 
Fon-der-Flaass proved that $\ord(\fX_{\eus P})=k+m$, see~\cite[Theorem\,2]{flaas1}. This confirms 
Conjecture~\ref{conj:general2}(i) for all $1$-standard $\BZ$-gradings with $\Delta(1)\simeq
\eus C_k\times \eus C_m$, because it is then clear that the highest weight of $\Delta(1)$ (=\,the root of maximal height) has height
$k+m-1$. In particular, this happens for all
but one $1$-standard $\BZ$-gradings in types $\GR{A}{n}$,
$\GR{B}{n}$, and $\GR{C}{n}$. More precisely, 

\textbullet \quad If $\Pi(1)=\{\ap_i\}$ for $\GR{A}{n}$, then 
$\Delta(1)\simeq \eus C_{i}\times \eus C_{n+1-i}$ and the highest weight in $\Delta(1)$ is $\theta$.

\textbullet \quad If $\Pi(1)=\{\ap_i\}$ for $\GR{B}{n}$, then 
$\Delta(1)\simeq \eus C_{i}\times \eus C_{2n+1-2i}$.

\textbullet \quad If $\Pi(1)=\{\ap_i\}$ for $\GR{C}{n}$ and $i<n$, then 
$\Delta(1)\simeq \eus C_{i}\times \eus C_{2n-2i}$.
\\
However, $\Delta(1)$ has another structure for the remaining case in $\GR{C}{n}$ and all
$1$-standard gradings of $\GR{D}{n}$. If $\Pi(1)=\{\ap_i\}$ for $\GR{D}{n}$ and $i\le n-2$, then
$\Delta(1)\simeq \eus C_{i}\times \eus D_{n-i}$. Note, however, that $\eus D_2\simeq
\eus C_2\times \eus C_2$.
\end{ex}

\begin{ex}  \label{ex:3-chains} 
For the poset $\eus C_k\times \eus C_m\times \eus C_n=:\eus C_{(k,m,n)}$, it is proved 
in~\cite[Theorem\,6(b)]{flaas2}
that $\ord(\fX)=k+m+n-1$ if $k=2$. (The general case is open!) 
In real life, posets of the form $\eus C_{(2,m,n)}$ occur in connection 
with $1$-standard $\BZ$-gradings corresponding to the branching node of the Dynkin diagram for 
$\GR{D}{n}$ or $\GR{E}{n}$. Namely,

\textbullet \quad If $\Pi(1)=\{\ap_{n-2}\}$ for $\GR{D}{n}$, then 
$\Delta(1)\simeq \eus C_{(2,2,n-2)}$;

\textbullet \quad If $\Pi(1)=\{\ap_3\}$ for $\GR{E}{6}$, then 
$\Delta(1)\simeq \eus C_{(2,3,3)}$;

\textbullet \quad If $\Pi(1)=\{\ap_4\}$ for $\GR{E}{7}$, then 
$\Delta(1)\simeq \eus C_{(2,3,4)}$;

\textbullet \quad If $\Pi(1)=\{\ap_5\}$ for $\GR{E}{8}$, then 
$\Delta(1)\simeq \eus C_{(2,3,5)}$.
\\
This confirms Conjecture~\ref{conj:general2}(i) for all these cases. The marked Dynkin diagrams for 
$\GR{D}{5}$ and  $\GR{E}{n}$ are depicted below (the black node represents the simple root in $\Pi(1)$):

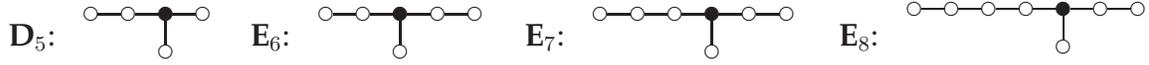
\quad $\GR{D}{5}$: 
\begin{picture}(55,30)(-5,5)
\setlength{\unitlength}{0.013in}
\put(35,18){\circle*{5}}
\put(35,3){\circle{5}}
\multiput(5,18)(15,0){4}{\circle{5}}
\multiput(8,18)(15,0){3}{\line(1,0){9}}
\put(35,6){\line(0,1){9}}
\end{picture}
\quad
$\GR{E}{6}$: 
\begin{picture}(70,30)(-5,5)
\setlength{\unitlength}{0.013in}
\put(35,18){\circle*{5}}
\put(35,3){\circle{5}}
\multiput(5,18)(15,0){5}{\circle{5}}
\multiput(8,18)(15,0){4}{\line(1,0){9}}
\put(35,6){\line(0,1){9}}
\end{picture}
\quad
$\GR{E}{7}$: 
\begin{picture}(85,30)(-5,5)
\setlength{\unitlength}{0.013in}
\put(50,18){\circle*{5}}
\put(50,3){\circle{5}}
\multiput(5,18)(15,0){6}{\circle{5}}
\multiput(8,18)(15,0){5}{\line(1,0){9}}
\put(50,6){\line(0,1){9}}
\end{picture}
\quad
$\GR{E}{8}$: 
\begin{picture}(100,30)(-5,3)
\setlength{\unitlength}{0.013in}
\put(65,18){\circle*{5}}
\put(65,3){\circle{5}}
\multiput(5,18)(15,0){7}{\circle{5}}
\multiput(8,18)(15,0){6}{\line(1,0){9}}
\put(65,6){\line(0,1){9}}
\end{picture}
\vskip.7ex\noindent
By the general rule, the subdiagram of white nodes represents $\g(0)$, and the bonds through the black node determine the $\g(0)$-module $\g(1)$, see Remark~\ref{rem:g1-dostat}.
\end{ex}
\begin{rmk}    \label{rmk:plane-partit}
An upper (lower) ideal in $\eus C_{(k,m,n)}$ can be identified with a plane partition with at most 
$k$ rows, at most $m$ columns, and with each entry $\le n$. Therefore, in case of $\eus C_{(k,m,n)}$,
our $\eus M$-polynomial is nothing but the rank-generating function for such plane partitions.
In \cite[Theorem\,11.2]{andrews}, one finds a closed formula for that generating function, which goes back to MacMahon.
Letting $(t)_r=(1-t)(1-t^2)\ldots(1-t^r)$, one has
\beq    \label{eq:formula-andrews}
\eus M_{\eus C_{(k,m,n)}}=\frac{(t)_1(t)_2\ldots(t)_{k-1}(t)_{m+n}(t)_{m+n+1}\ldots(t)_{m+n+k-1}}{
(t)_m(t)_{m+1}\ldots(t)_{m+k-1}(t)_n(t)_{n+1}\ldots(t)_{n+k-1}} .
\eeq
Substituting $t=1$, one obtains
\[
     \#\AN(\eus C_{(k,m,n)})=\frac{1!2!\dots(k-1)!(m+n)!(m+n+1)!\dots (m+n+k-1)!}{
   m!(m+1)!\dots (m+k-1)!n!(n+1)!\dots (n+k-1)!    } .
\]
The last formula appears in \cite[p.\,553]{flaas2}, where the relationship between plane partitions and 
antichains in $\eus C_{(k,m,n)}$ is also alluded to. In \cite[p.81]{macd95}, one finds the 
rank-generating function for $\eus C_{(k,m,n)}$ exactly in the form suggested by 
Eq.~\eqref{eq:arb-wmf-suggest}.
\end{rmk}

\begin{rmk}   \label{rmk:average-N}
If the average value of the size of antichains in all $\xdo$-orbits is the same, then it must be equal to
the average value of the size of {\bf all} antichains in $\Delta(1)$. By the very definition of $\ndt$, the 
latter average value equals $\eus N'_{\Delta(1)}(1)/\eus N_{\Delta(1)}(1)$.
Therefore, if Conjecture~\ref{conj:general2}(i),(ii) is true, then one must have the equality
\[
    \frac{\eus N'_{\Delta(1)}(1)}{\eus N_{\Delta(1)}(1)}=\frac{\#\Delta(1)}{d_1+1} .
\]
Our previous calculations show that this is really the case for the abelian and extra-special gradings, see
Eq.~\eqref{eq:always-true-ab} and \eqref{eq:always-true-extra}.
\end{rmk}
\begin{rmk}   \label{rmk:average-M}
If the average value of the size of upper ideals in all $\xdo$-orbits is the same, then it must be equal to
the average value of the size of {\bf all} upper ideals in $\Delta(1)$. By the very definition of $\mdt$, the 
latter average value equals $\eus M'_{\Delta(1)}(1)/\eus M_{\Delta(1)}(1)$. Since $\mdt$ is palindromic,
the last fraction is equal to $\frac{1}{2}\deg\mdt=\frac{1}{2}\dim\g(1)$. This explains the average value
in Conjecture~\ref{conj:general2}(iii).
\end{rmk}

We proved that  $\#\anod=\#\Pi_l{\cdot}(h-1)$ in the extra-special case (Theorem~\ref{thm:extra-case}(ii)). 
Combined with the conjectural value $\ord(\xdo)=h-1$ and some explicit calculations,
this suggests the following:

\begin{conj}   \label{conj:extra}
In the extra-special case, 
the number of\/ $\fX_{\Delta(1)}$-orbits equals $\#\Pi_l$, and each orbit is of size $h-1$.
Furthermore, if $h$ is even (which only excludes the case of $\GR{A}{2k}$, where $h=2k+1$), 
then each $\xdo$-orbit contains a unique Lagrangian upper ideal.
\end{conj}

\noindent
This conjecture is readily verified for $\GR{C}{n}$ or $\GR{G}{2}$, where $\#\Pi_l=1$, and the case 
of $\GR{A}{n}$ is easy.  It is also possible to perform all necessary calculations by hand for 
$\GR{F}{4}$ and $\GR{E}{6}$. 

Our formulae for $\eus N$-polynomials in Sections~\ref{sect:ab-case} and \ref{sect:ext-case} show 
that they are not always palindromic. Clearly, if $\eus N_\eus P(t)$ is palindromic, then it is monic, 
hence $\eus P$ has a unique antichain of maximal size. Therefore, if $\eus P$ is Sperner 
and  $\eus N_\eus P(t)$ is palindromic, then $\eus P$ has a unique rank level of maximal size.
It is plausible that, for the weight posets of the form $\Delta(1)$, this necessary condition is also sufficient:

\begin{conj}    \label{conj:n-palindrom}
Let $\g=\bigoplus_{i\in\BZ}\g(i)$ be a $1$-standard $\BZ$-grading. Then $\ndt$ is palindromic if and only
if $\Delta(1)$ has a unique rank level of maximal size.  
\end{conj}

\noindent
Our formulae confirm this conjecture in the abelian and extra-special cases.

\begin{ex}   \label{ex:boolean-cubes}
($\BZ$-gradings versus periodic gradings)
\par
1) For $\g=\mathfrak{so}_8$ (type $\GR{D}{4}$), the extra-special $\BZ$-grading corresponds
to the branching node of the Dynkin diagram, i.e., $\Pi(1)=\{\ap_2\}$. Hence
$\widetilde{\g(0)}=(\tri)^3$ and  $\g(1)$ is the tensor product of the standard 
representations of three copies of $\tri$. Therefore,  $\Delta(1)$ is the 3-dimensional Boolean algebra 
$\mathbb B^3 \simeq \eus C_{2,2,2}$. Here $\#\AN(\mathbb B^3)=20$,
$\eus H_{\mathbb B^3}(t)=1+8t+9t^2+2t^3$, and $h-1=5$. In fact, all the relevant assertions of the above conjectures are satisfied here. In particular, $\ord(\fX_{\mathbb B^3})=5$, 
$\AN(\mathbb B^3)$ consists of four $\fX_{\mathbb B^3}$-orbits of size $5$, and both average values are constant, as prescribed.
\par
On the other hand, $\mathfrak{so}_8$ has a $\BZ_2$-grading such that $\g_0=(\tri)^4$ and 
$\g_1$ is the tensor product of the standard representations of four copies of
$\tri$. Hence $\Delta_1$ is the 4-dimensional Boolean algebra $\mathbb B^4$.
The poset $\mathbb B^4$ has five rank levels and the corresponding level 
cardinalities $\#(\mathbb B^4)_i$ are $1,4,6,4,1$. Here the hypothetical product formula for
$\#\AN(\Delta_1)$, see Eq.~\eqref{eq:arb-wmf-suggest}, yields the output 
\[
      \frac{2^1 3^4 4^6 5^4 6^1}{1^1 2^4 3^6 4^4 5^1 }=500/3 =166,66... \ , 
\]
which is absurd.
Furthermore, the hypothetical product formula for $\eus M_{\Delta_1}(t)$ gives a rational function 
that is not a polynomial. Thus, one cannot always expect the validity of Eq.~\eqref{eq:arb-wmf-suggest} 
if $V$ is not associated with a $\BZ$-grading.   
\par 
It is known that $\#\AN(\mathbb B^4)=168$ and the sizes of $\fX_{\mathbb B^4}$-orbits are
$2,3,6$, see e.g.~\cite[p.\,73]{deza}. Therefore $\ord(\fX_{\mathbb B^4})=6$, which agrees with
Conjecture~\ref{conj:general2}(i). But, the average value
of sizes of antichains along $\fX_{\mathbb B^4}$-orbits is not constant here! Let us regard elements of
$\mathbb B^4$ as subsets of $\{1,2,3,4\}$. For the orbit of $\Gamma_1=\{\varnothing\}$, 
the above average value is $8/3$, while for the orbit of $\Gamma_2=
\{\{1,2\},\{3,4\}\}$, the average value is $3$. Furthermore, explicit computations show that
\[
    \eus N_{\mathbb B^4}(t)=1+16t+55t^2+64t^3+25t^4+6t^5+t^6 .
\]
That is, whereas $\mathbb B^4$ has a unique rank level of maximal size, $\eus N_{\mathbb B^4}(t)$ is not palindromic.

2) The choice of the branching node $\ap_i$ in the {\bf extended} Dynkin diagram $\GRt{E}{n}$ ($n=6,7,8$) provides a $\BZ_d$-grading with $d=[\theta:\ap_i]$ such that $\g_0=
\mathfrak{sl}_k\times\mathfrak{sl}_m\times\mathfrak{sl}_n$ and 
$\g_1=\sfr(\varpi_1)\otimes\sfr(\varpi'_1)\otimes \sfr(\varpi''_1)$. Therefore, $\Delta_1\simeq
\eus C_{(k,m,n)}$. Here $(k,m,n)=(3,3,3)$, $(2,4,4)$, $(2,3,6)$ for 
$\GR{E}{6}$, $\GR{E}{7}$, $\GR{E}{8}$, respectively, see pictures below. The respective values of
$d$ are $3,4,6$.

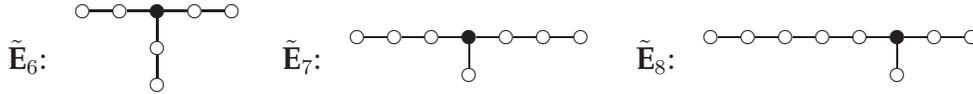
\qquad $\GRt{E}{6}$: 
\begin{picture}(70,40)(-5,-5)
\setlength{\unitlength}{0.013in}
\put(35,18){\circle*{5}}
\multiput(35,-12)(0,15){2}{\circle{5}}
\multiput(5,18)(15,0){5}{\circle{5}}
\multiput(8,18)(15,0){4}{\line(1,0){9}}
\multiput(35,-9)(0,15){2}{\line(0,1){9}}
\end{picture}
\quad
$\GRt{E}{7}$: 
\begin{picture}(100,30)(-5,5)
\setlength{\unitlength}{0.013in}
\put(50,18){\circle*{5}}
\put(50,3){\circle{5}}
\multiput(5,18)(15,0){7}{\circle{5}}
\multiput(8,18)(15,0){6}{\line(1,0){9}}
\put(50,6){\line(0,1){9}}
\end{picture}
\quad
$\GRt{E}{8}$: 
\begin{picture}(115,30)(-5,5)
\setlength{\unitlength}{0.013in}
\put(80,18){\circle*{5}}
\put(80,3){\circle{5}}
\multiput(5,18)(15,0){8}{\circle{5}}
\multiput(8,18)(15,0){7}{\line(1,0){9}}
\put(80,6){\line(0,1){9}}
\end{picture}
\vskip1.5ex\noindent
By the general rule, the subdiagram of white nodes represents $\g_0$, and the bonds through the black node determine the $\g_0$-module $\g_1$, see \cite[\S\,8]{vi76}, \cite[Ch.\,3,\,\S 3.7]{t41} for details.
 \par
 As explained in Remark~\ref{rmk:plane-partit}, Conjecture~\ref{conj:general1} 
 (or Eq.~\eqref{eq:arb-wmf-suggest}) does provide here the $\eus M$-polynomial of $\Delta_1$.
However, at least in the $\GR{E}{6}$-case, one detects two $\fX_{\Delta_1}$-orbits with different values for the average size of antichains. Thus, Conjecture~\ref{conj:general2}(ii) fails there.
\end{ex}

\begin{ex}   \label{ex:Z-grad-outside}
For the 1-standard $\BZ$-grading of $\GR{E}{8}$ with $\Pi(1)=\{\ap_8\}$, we have $\widetilde{\g(0)}=\GR{A}{7}=\mathfrak{sl}_8$ and
$\g(1)=\sfr(\varpi_3)$, the third fundamental representation of  $\mathfrak{sl}_8$. Then $\dim\g(1)=\genfrac{(}{)}{0pt}{}{8}{3}=56$ 
and the poset $\Delta(1)=\eus P(\varpi_3)$ has 16 rank levels whose sizes are\\
$1,\,1,\,2,\,3,\,4,\,5,\,6,\,6,\,6,\,6,\,5,\,4,\,3,\,2,\,1,\,1$. Here the  formulae of Conjecture~\ref{conj:general1}
provide a polynomial in $t$ with nonnegative coefficients and an integer! That is, conjecturally, it should
be true that $\#\anod=11{\cdot}13{\cdot}17=2431$ 
and $\ord(\fX)=17$. 
\end{ex}

{\it \bfseries Concluding remarks.}
In this section, we formulated a bunch of conjectures for various properties of weight posets.
I have every confidence in validity of all these conjectures for the weight posets of the form $\Delta(1)$.
Moreover, some of the conjectures may have a wider range of applicability. 
For instance, the posets $\eus C_{(k,m,n)}$ are related to  periodic or $\BZ$-gradings only for a limited set of triples $(k,m,n)$. (All those instances essentially appear in Examples~\ref{ex:3-chains} and \ref{ex:boolean-cubes}(2).) But the formula for the $\eus M$-polynomial in Conjecture~\ref{conj:general1}
applies to all such posets, see Remark~\ref{rmk:plane-partit}. 
\par Furthermore, it is likely that, for all posets
$\eus P=\eus C_{(k,m,n)}$, the order of $\fX_\eus P$ equals $k+m+n-1$, which agrees with Conjecture~\ref{conj:general2}(i). (At least, this is proved for $k=2$ in \cite{flaas2}.)
\par
On the other hand, the property on the average value of the size of antichains in $\fX$-orbits 
(see Conjecture~\ref{conj:general2}(ii)) seems to be the most vulnerable one. Our examples suggest that it fails once we leave the variety of weight posets related to $\BZ$-gradings.

\vskip1ex
{\small  {\bf Acknowledgements.}
This is an expanded version of my talk at the workshop ``Non-crossing partitions'' (Bielefeld, June 
2014). I would like to thank the organisers for possibility to speak there. While preparing my talk,
I dug up my old notes on the subject, made several new observations, and even produced some proofs.
}

\end{document}